\documentclass[a4paper,reqno,10pt]{amsart}

\usepackage{epsf,graphicx,epsfig}
\usepackage{amscd}
\usepackage{amsmath,latexsym,amssymb,amsthm}
\usepackage[nospace,noadjust]{cite}
\usepackage{textcomp}
\usepackage{setspace,cite}
\usepackage{lscape,fancyhdr,fancybox}
\usepackage{stmaryrd}
\usepackage[all,cmtip]{xy}
\usepackage{tikz}
\usepackage{cancel}
\usetikzlibrary{shapes,arrows,decorations.markings}
\usepackage{graphicx}

\usepackage{color}
\usepackage{url}
\usepackage{enumerate}
\usepackage[mathscr]{euscript}

  \theoremstyle{plain}

\swapnumbers
    \newtheorem{thm}{Theorem}[section]
    \newtheorem{prop}[thm]{Proposition}

    \newtheorem{subsec}[thm]{}
\theoremstyle{definition}
    \newtheorem{defn}[thm]{Definition}
        \newtheorem{remark}[thm]{Remark}
    \newtheorem{exam}[thm]{Example}

\theoremstyle{remark}

\title{}
\author{}
\date{}
\usepackage{amssymb}

\usepackage{hyperref}
\hypersetup{
	colorlinks,
	citecolor=blue,
	filecolor=black,
	linkcolor=blue,
	urlcolor=black
}

\begin{document}

\title[Compatible $\mathcal {O}$-operators and compatible dendriform algebras]{$L_\infty$-structures and cohomology theory of compatible $\mathcal {O}$-operators and compatible dendriform algebras}



\author{Apurba Das}
\address{Department of Mathematics,
Indian Institute of Technology, Kharagpur 721302, West Bengal, India.}
\email{apurbadas348@gmail.com}

\author{Shuangjian Guo}
\address{School of Mathematics and Statistics, Guizhou University of Finance and Economics, Guiyang  550025, China.} 
\email{shuangjianguo@126.com}

\author{Yufei Qin}
\address{School of Mathematical Sciences,  Shanghai Key Laboratory of PMMP,  East China Normal University, Shanghai 200241, China.}
\email{290673049@qq.com}

\maketitle







\begin{abstract}
The notion of $\mathcal{O}$-operator is a generalization of Rota-Baxter operator in the presence of a bimodule over an associative algebra. A compatible $\mathcal{O}$-operator is a pair consisting of two $\mathcal{O}$-operators satisfying a compatibility relation. A compatible $\mathcal{O}$-operator algebra is an algebra together with a bimodule and a compatible $\mathcal{O}$-operator. In this paper, we construct a graded Lie algebra and an $L_\infty$-algebra that respectively characterize compatible $\mathcal{O}$-operators and compatible $\mathcal{O}$-operator algebras as Maurer-Cartan elements. Using these characterizations, we define cohomology of these structures and as applications we study formal deformations of compatible $\mathcal{O}$-operators and compatible $\mathcal{O}$-operator algebras. 
Finally, we consider a brief cohomological study of compatible dendriform algebras and find their relation with the cohomology of compatible associative algebras and compatible $\mathcal{O}$-operators.
\end{abstract}

\medskip

\medskip



\medskip


\thispagestyle{empty}

\tableofcontents



\section{Introduction}

Rota-Baxter operators are algebraic abstraction of the integral operator. They first appeared in a work of Baxter in the fluctuation theory of probability \cite{B60}. Subsequently, such operators were studied by Rota \cite{R-rota}, Cartier \cite{cartier}, Atkinson \cite{atkinson} among others. Rota-Baxter operators appeared in many contexts of mathematics and mathematical physics. For instance, they appeared in the combinatorial study of shuffle algebras \cite{guo-free-rota}, splitting of algebras and operads \cite{BBGN13}, theory of bialgebras \cite{aguiar-pre-dend} and renormalizations in quantum field theory \cite{CK00}. In \cite{U08} Uchino introduced a notion of generalized Rota-Baxter operator, also known as $\mathcal{O}$-operator or relative Rota-Baxter operator (mostly known as $\mathcal{O}$-operator in the literature \cite{bgn,D20}). Let $A$ be an associative algebra and $M$ be an $A$-bimodule. We write $~\cdot~$ to denote the associative multiplication on $A$ as well as both the left and right $A$-actions on $M$. A linear map $T: M \rightarrow A$ is an $\mathcal{O}$-operator on $M$ over the algebra $A$ if the map $T$ satisfies
\begin{align*}
T(u) \cdot T(v) = T (T(u) \cdot v + u \cdot T(v)), ~ \text{for } u, v \in M.
\end{align*}
A Rota-Baxter operator on $A$ is an $\mathcal{O}$-operator on $A$ (consider as the adjoint bimodule space) over the algebra $A$ itself. Uchino observed that $\mathcal{O}$-operators can be seen as a noncommutative analogue of Poisson structures. More precisely, a skew-symmetric element $r \in A \otimes A$ is an associative Yang-Baxter solution (noncommutative analogue of Poisson tensor) if and only if the map $r^\sharp : A^* \rightarrow A$ is an $\mathcal{O}$-operator on $A^*$ (consider as the coadjoint bimodule space) over the algebra $A$. He also observed that an $\mathcal{O}$-operator induces a dendriform algebra structure. Rota-Baxter operators and $\mathcal{O}$-operators are also well studied in the context of Lie algebras \cite{S83,kuper}. Recently, cohomology and deformation theory of Rota-Baxter operators and $\mathcal{O}$-operators are settled in \cite{shengO,D20}. See \cite{G12} for more results on $\mathcal{O}$-operators.

\medskip

In the mathematical study of biHamiltonian mechanics, two Poisson structures on a manifold are said to be compatible if their sum also defines a Poisson structure \cite{yks-mag}. Compatible Poisson structures are closely related to compatible Lie algebras via dualization \cite{dual1,dual2}. Our aim in this paper is to study the noncommutative analogue of compatible Poisson structures. They are given by compatible $\mathcal{O}$-operators introduced in \cite{LSB19}. A compatible $\mathcal{O}$-operator is a pair $(T_1, T_2)$ in which $T_1, T_2$ are both $\mathcal{O}$-operators satisfying an additional compatibility which is equivalent to the fact that $T_1 + T_2$ is also an $\mathcal{O}$-operator. We call the triple $(A,M,(T_1,T_2))$ consisting of an algebra $A$, an $A$-bimodule $M$ and a compatible $\mathcal{O}$-operator $(T_1, T_2)$ as a `compatible   $\mathcal{O}$-operator algebra'. We introduce (skew-symmetric) compatible associative Yang-Baxter solutions and show that they induce compatible $\mathcal{O}$-operators. 

\medskip

Next, we focus on the cohomology of compatible $\mathcal{O}$-operators and compatible $\mathcal{O}$-operator algebras. Given an associative algebra $A$ and an $A$-bimodule $M$, we first construct a graded Lie algebra whose Maurer-Cartan elements are precisely compatible $\mathcal{O}$-operators. Using this characterization, we define the cohomology of a compatible $\mathcal{O}$-operator $(T_1, T_2)$. On the other hand, we observed that a compatible $\mathcal{O}$-operator induces a compatible associative algebra and a compatible bimodule over it. The cohomology of the induced compatible associative algebra is isomorphic to the cohomology of the compatible $\mathcal{O}$-operator $(T_1, T_2)$. We show that there is a morphism from the cohomology of a compatible $\mathcal{O}$-operator $(T_1,T_2)$ to the cohomology of the $\mathcal{O}$-operator $T_1 + T_2$. As an application of our cohomology, we study formal deformations of the compatible $\mathcal{O}$-operator $(T_1,T_2)$ by keeping the underlying algebra $A$ and the $A$-bimodule $M$ intact. 
We also consider finite order deformations of a compatible $\mathcal{O}$-operator. Given a finite order deformation, we associate a $2$-cocycle, whose cohomology class is called the `obstruction class'. The given deformation is extensible to a deformation of next order if and only if the corresponding obstruction class vanishes. Next, following the approach of \cite{DasMishra} in the context of compatible structures, we construct an $L_\infty$-algebra whose Maurer-Cartan elements are compatible $\mathcal{O}$-operator algebras. Hence we may define the cohomology of a compatible $\mathcal{O}$-operator algebra $(A,M, (T_1,T_2))$. Using such cohomology, we study formal deformations of the compatible $\mathcal{O}$-operator algebra $(A,M, (T_1,T_2))$, where we allow to deform the underlying algebra $A$, the $A$-bimodule $M$ and the compatible $\mathcal{O}$-operator $(T_1, T_2)$. 



\medskip

In the last part, we consider compatible dendriform algebras introduced in \cite{LSB19}. Like $\mathcal{O}$-operators are related with dendriform algebras, compatible $\mathcal{O}$-operators are related with compatible dendriform algebras.
By following the cohomology of dendriform algebras given in \cite{Das11802,lod-val}, here we define cohomology of compatible dendriform algebras. Our cohomology of compatible dendriform algebras can be seen as a splitting of the cohomology of compatible associative algebras recently introduced in \cite{cdm}. Finally, we find a relation among the cohomology of a compatible $\mathcal{O}$-operator to the cohomology of the induced compatible dendriform algebra.

\medskip

The paper is organized as follows. In Section \ref{sec-comp-o}, we consider compatible $\mathcal{O}$-operators and find their relation with compatible associative Yang-Baxter solutions. Maurer-Cartan characterization and cohomology theory of compatible $\mathcal{O}$-operators and compatible $\mathcal{O}$-operator algebras are defined in Section \ref{sec-coho}. In Section \ref{sec-def}, we study formal deformation theory of compatible $\mathcal{O}$-operators and compatible $\mathcal{O}$-operator algebras. Finally, cohomology theory of compatible dendriform algebras, and their relations with some known cohomologies are given in Section \ref{sec-comp-dend}.

\medskip

Throughout this paper, all (graded) vector spaces, linear and multilinear maps, tensor products are over a field {\bf k} of characteristic $0$.

\medskip

\medskip

\section{Compatible $\mathcal{O}$-operators and compatible associative Yang-\\Baxter solutions}\label{sec-comp-o}
In this section, we first consider compatible $\mathcal{O}$-operators introduced in \cite{LSB19} and then define compatible associative Yang-Baxter solutions. We show that (skew-symmetric) compatible associative Yang-Baxter solutions give rise to compatible $\mathcal{O}$-operators.

\begin{defn}
Let $A$ be an associative algebra and $M$ be an $A$-bimodule.
A compatible $\mathcal{O}$-operator on $M$ over the algebra $A$ consists of a pair $(T_1, T_2)$ of $\mathcal{O}$-operators satisfying the compatibility
\begin{align}
T_1(u) \cdot T_2(v) + T_2(u) \cdot T_1(v)= T_1(T_2(u) \cdot v + u \cdot T_2(v)) + T_2 (T_1(u) \cdot v + u \cdot T_1(v)), \text{ for } u, v \in M.
\end{align}
A triple $(A, M, (T_1, T_2))$ consisting of an algebra $A$, an $A$-bimodule $M$ and a compatible $\mathcal{O}$-operator $(T_1, T_2)$ on $M$ over the algebra $A$ is called a `compatible $\mathcal{O}$-operator algebra'. As $\mathcal{O}$-operator algebras are also called relative Rota-Baxter algebras, hence `compatible $\mathcal{O}$-operator algebras' may be called `compatible relative Rota-Baxter algebras'.
\end{defn}

It follows from the above compatibility that $\lambda T_1 + \eta T_2$ is an $\mathcal{O}$-operator, for any $\lambda, \eta \in {\bf k}$. In particular, the sum $T_1 + T_2$ is an $\mathcal{O}$-operator, called the `total $\mathcal{O}$-operator'.

\begin{exam}
Let $T:M\rightarrow A$  be an $\mathcal{O}$-operator.  Then $\widetilde{T} := -T$ is also an $\mathcal{O}$-operator and the pair $(T, \widetilde{T})$ is a compatible $\mathcal{O}$-operator.
\end{exam}

\begin{exam}
Let $A =C^1([0, 1])$ and $M=C^0([0, 1])$. Then $A$ is an associative algebra and $M$ be an $A$-bimodule. Let $T_1, T_2 : M \rightarrow A$ be defined by
\begin{eqnarray*}
 T_1(f)(x)=\int^x_0 k_1(t)f(t) dt ~~ \text{ and } ~~ T_2(f)(x)=\int^x_0 k_2(t)f(t) dt,
\end{eqnarray*}
for some functions $k_1$ and $k_2$ in $C^0([0,1])$.
Then the pair $(T_1, T_2)$ is a compatible $\mathcal{O}$-operator.
\end{exam}

\begin{exam}
Let $d_1, d_2 : A \rightarrow M$ be two derivations that are invertible and satisfying the compatibility
\begin{align*}
d_1^{-1} d_2 (a \cdot b) + d_2^{-1} d_1 (a' \cdot b') = a \cdot b' + a' \cdot b, \text{ for } a, b \in A.
\end{align*}
Then it is easy to check that $(T_1 = d_1^{-1}, T_2 = d_2^{-1})$ is a compatible $\mathcal{O}$-operator.
\end{exam}

\begin{defn}
Let $(T_1, T_2)$ and $(T_1', T_2')$ be two compatible $\mathcal{O}$-operators on $M$ over the algebra $A$. A morphism of compatible $\mathcal{O}$-operators from $(T_1, T_2)$ to $(T_1', T_2' )$  is given by a pair $(\phi, \psi)$ of an algebra morphism $\phi : A \rightarrow A$ and a linear map $\psi : M \rightarrow M$ satisfying
\begin{align*}
\phi \circ T_1 = T'_1 \circ \psi, ~~
\phi \circ T_2 = T'_2 \circ \psi, ~~
\psi(a \cdot u) =  \phi(a) \cdot \psi(u) ~~ \text{ and } ~~
\psi(u \cdot a) = \psi(u) \cdot \phi(a), \text{ for } a \in A, u \in M.
\end{align*}
\end{defn}

\medskip

It is well-known \cite{U08} that an $\mathcal{O}$-operator $T: M \rightarrow A$ induces an associative algebra structure on $M$  given by
\begin{align*}
u \star v = T(u) \cdot v + u \cdot T(v), \text{ for } u, v \in M.
\end{align*}
Moreover, there is an $M$-bimodule structure on $A$ with {left and right} $M$-actions
\begin{align*}
l_T : M \otimes A \rightarrow A, ~ l_T (u, a) = T(u) \cdot a - T (u \cdot a) ~~ \text{ and } ~~
r_T : A \otimes M \rightarrow A, ~ r_T (a, u) = a \cdot T(u) - T ( a \cdot u).
\end{align*}

\medskip

In the following, we will generalize these results in the context of compatible $\mathcal{O}$-operators. First, recall that a compatible associative algebra is a triple $(A, \cdot_1, \cdot_2)$ in which $(A, \cdot_1)$ and $(A, \cdot_2)$ are both associative algebras satisfying additionally
\begin{align}\label{ass-comp}
(a \cdot_1 b ) \cdot_2 c + (a \cdot_2 b ) \cdot_1 c = a \cdot_2 (b \cdot_1 c ) + a \cdot_1 (b \cdot_2 c ), \text{ for } a, b, c \in A.
\end{align}
Like an associative algebra gives rise to a Lie algebra via skew-symmetrization, a compatible associative algebra gives rise to a compatible Lie algebra via skew-symmetrization of both associative products. See \cite{cdm} for further relation between compatible associative and Lie algebras.

Let $(A, \cdot_1, \cdot_2)$ be a compatible associative algebra. A compatible bimodule over it consists of a quintuple $(M, l_1, r_1, l_2, r_2)$ of a vector space $M$ with four bilinear operations
\begin{align*}
l_1 : A \otimes M &\rightarrow M,  \quad \quad r_1 : M \otimes A \rightarrow A, \quad \quad l_2 : A \otimes M \rightarrow M, \quad \quad r_2 : M \otimes A \rightarrow A \\
(a,u) &\mapsto a \cdot_1 u  ~~~~ \quad \qquad \qquad ~ (u,a) \mapsto u \cdot_1 a ~~~ \quad \qquad (a,u) \mapsto a \cdot_2 u    \qquad \qquad  (u,a) \mapsto u \cdot_2 a
\end{align*}
so that $(M, l_1, r_1)$ is a bimodule over the associative algebra $(A, \cdot_1)$, and $(M, l_2, r_2)$ is a bimodule over the associative algebra $(A, \cdot_2)$ satisfying additionally for any $a, b \in A$ and $u \in M$,
\begin{align*}
( a \cdot_1 b ) \cdot_2 u  + ( a \cdot_2 b ) \cdot_1 u  =~& a \cdot_1 ( b \cdot_2 u) + a \cdot_2 ( b \cdot_1 u),\\
( a \cdot_1 u ) \cdot_2 b  + ( a \cdot_2 u ) \cdot_1 b  =~& a \cdot_1 ( u \cdot_2 b) + a \cdot_2 ( u \cdot_1 b),\\
( u \cdot_1 a ) \cdot_2 b  + ( u \cdot_2 a ) \cdot_1 b  =~& u \cdot_1 ( a \cdot_2 b) + u \cdot_2 ( a \cdot_1 b).
\end{align*}
These three compatibility conditions is equivalent to the fact that $(M, l_1 + l_2, r_1 + r_2)$ is a bimodule over the associative algebra $(A, \cdot_1 + \cdot_2)$.
Note that any compatible associative algebra $(A, \cdot_1, \cdot_2)$ is a bimodule over itself with $l_1 = r_1 = \cdot_1$ and $l_2 = r_2 = \cdot_2$. This is called adjoint compatible bimodule.


\begin{thm}\label{induced-comp-th}
Let $A$ be an associative algebra and $M$ be an $A$-bimodule. Let $(T_1, T_2)$ be a compatible $\mathcal{O}$-operator on $M$ over the algebra $A$. Then $(M, \star_1, \star_2)$ is a compatible associative algebra, where
\begin{align*}
u \star_1 v = T_1(u) \cdot v + u \cdot T_1 (v) ~~~ \text{ and } ~~~ u \star_2 v = T_2(u) \cdot v + u \cdot T_2 (v), \text{ for } u, v \in M.
\end{align*}
Moreover, $(A, l_{T_1}, r_{T_1}, l_{T_2}, r_{T_2})$ is a compatible bimodule over the compatible associative algebra $(M, \star_1, \star_2)$.
\end{thm}

\begin{proof}
For any $u, v \in M$, we observe that
\begin{align}\label{sum-star}
u \star_1 v + u \star_2 v = (T_1 + T_2)(u) \cdot v + u \cdot (T_1 + T_2) (v).
\end{align}
Since $(T_1, T_2)$ is a compatible $\mathcal{O}$-operator, it follows that $T_1 + T_2$ is an $\mathcal{O}$-operator. Hence from (\ref{sum-star}), we conclude that $(M, \star_1 + \star_2)$ is an associative algebra. In other words, $(M, \star_1, \star_2)$ is a compatible associative algebra.

Similarly, for any $u \in M$ and $a \in A$, we see that
\begin{align*}
(l_{T_1} + l_{T_2}) (u, a) = l_{T_1 + T_2} (u, a) ~~~~ \text{ and } ~~~~ (r_{T_1} + r_{T_2}) (a,u) = r_{T_1 + T_2} ( a, u).
\end{align*}
This shows that $(A, l_{T_1}, r_{T_1}, l_{T_2}, r_{T_2})$ is a bimodule over the compatible associative algebra $(M, \star_1, \star_2)$.
\end{proof}

\medskip

\noindent {\bf Compatible associative Yang-Baxter solutions.} Let $A$ be an associative algebra. For any element $r =\sum r^{[1]}\otimes r^{[2]} \in  A\otimes A$, Aguiar \cite{A99} considers the associative Yang-Baxter equation on $A$ as
\begin{align}\label{aybe-eqn}
r^{13}r^{12}-r^{12}r^{23} + r^{23}r^{13} = 0,
\end{align}
or, more explicitly,
\begin{align*}
\sum r^{[1]}\cdot \widetilde{r}^{[1]}\otimes \widetilde{r}^{[2]}\otimes r^{[2]}-\sum r^{[1]}\otimes r^{[2]}\cdot \widetilde{r}^{[1]}\otimes \widetilde{r}^{[2]}+\sum  \widetilde{r}^{[1]}\otimes r^{[1]}\otimes r^{[2]}\cdot \widetilde{r}^{[2]}=0.
\end{align*}
A solution of (\ref{aybe-eqn}) is called an associative Yang-Baxter solution on $A$. If $r = \sum r^{[1]}\otimes r^{[2]} \in  A\otimes A$ is an associative Yang-Baxter solution, then the map $T: A \rightarrow A$ defined by $T(a) = \sum r^{[1]} \cdot a \cdot r^{[2]}$ is a Rota-Baxter operator on $A$ (i.e. an $\mathcal{O}$-operator on the adjoint bimodule space $A$ over the algebra $A$) \cite{A99,aguiar-pre-dend}. Later, Uchino \cite{U08} observed that skew-symmetric associative Yang-Baxter solutions  give rise to $\mathcal{O}$-operators on the coadjoint bimodule space $A^*$ over the algebra $A$. We will generalize these results in the context of compatible $\mathcal{O}$-operators.

\begin{defn}
Let $A$ be an associative algebra. A compatible associative Yang-Baxter solution on $A$ is a pair $(r_1, r_2)$, where $r_1 = \sum r_1^{[1]}\otimes r_1^{[2]}$ and $r_2 = \sum r_2^{[1]}\otimes r_2^{[2]}$ are both associative Yang-Baxter solutions on $A$ satisfying
\begin{align}
r_1^{13}r_2^{12} + r_2^{13}r_1^{12} -r_1^{12}r_2^{23} - r_2^{12}r_1^{23} + r_1^{23}r_2^{13} + r_2^{23}r_1^{13} = 0,
\end{align}
equivalently,
\begin{align}\label{comp-aybe}
\sum r_1^{[1]} \cdot r_2^{[1]} \otimes r_2^{[2]} \otimes r_1^{[2]} +~& \sum r_2^{[1]} \cdot r_1^{[1]} \otimes r_1^{[2]} \otimes r_2^{[2]} - \sum r_1^{[1]} \otimes r_1^{[2]} \cdot r_2^{[1]} \otimes r^{[2]}_2 - \sum r_2^{[1]} \otimes r_2^{[2]} \cdot r_1^{[1]} \otimes r^{[2]}_1 \nonumber \\
+~& \sum  r_2^{[1]} \otimes r_1^{[1]} \otimes r_1^{[2]} \cdot r_2^{[2]} + \sum r_1^{[1]} \otimes r_2^{[1]} \otimes r_2^{[2]} \cdot r_1^{[2]} = 0.
\end{align}
\end{defn}
\begin{exam}
 If $r = \sum r^{[1]}\otimes r^{[2]} \in  A\otimes A$ is an associative Yang-Baxter solution,  then the pair $(r, -r)$ is a  compatible associative Yang-Baxter solution on $A$.
\end{exam}

\begin{prop}
Let $A$ be an associative algebra and $(r_1, r_2)$ be a compatible associative Yang-Baxter solution on $A$. Then the pair $(T_1,T_2)$ of linear maps $T_1, T_2 : A \rightarrow A$ forms a compatible Rota-Baxter operator on $A$ (i.e. compatible $\mathcal{O}$-operator on the adjoint bimodule space $A$ over the algebra $A$), where $ T_1 (a) = \sum r_1^{[1]} \cdot a \cdot r_1^{[2]}$ and $ T_2 (a) = \sum r_2^{[1]} \cdot a \cdot r_2^{[2]}$, for $a \in A$.
\end{prop}

\begin{proof}
Note that, by replacing the first tensor product by $a$ and the second tensor product by $b$ in the identity (\ref{comp-aybe}), we get that
\begin{align*}
T_1 (T_2 (a) \cdot b) + T_2 (T_1 (a) \cdot b) - T_1 (a) \cdot T_2 (b) -  T_2 (a) \cdot T_1 (b) + T_2 (a \cdot T_1 (b) ) + T_1 (a \cdot T_2 (b) ) = 0.
\end{align*}
This proves the compatibility condition among $T_1$ and $T_2$. Hence $(T_1,T_2)$ is a compatible Rota-Baxter operator on $A$.
\end{proof}

Recall that an element $r = \sum r^{[1]} \otimes r^{[2]} \in A \otimes A$ is said to be skew-symmetric if $ r^\tau = - r $, where $r^\tau = \sum r^{[2]} \otimes r^{[1]}$. Then we have the following.


\begin{prop}
Let $A$ be an associative algebra and $(r_1, r_2)$ be a pair of skew-symmetric elements in $A \otimes A$. Let $r_1 = \sum r_1^{[1]} \otimes r_1^{[2]}$ and $r_2 = \sum r_2^{[1]} \otimes r_2^{[2]}$. Then $(r_1, r_2)$ is a compatible associative Yang-Baxter solution if and only if the pair $(r_1^\sharp, r_2^\sharp)$ of maps $r_1^\sharp, r_2^\sharp : A^* \rightarrow A$ forms a compatible $\mathcal{O}$-operator on the coadjoint bimodule space $A^*$ over the algebra $A$, where
\begin{align*}
r_1^\sharp (f) = \sum f (r_1^{[2]}) r_1^{[1]} ~~~ \text{ and } ~~~ r_2^\sharp (f) = \sum f (r_2^{[2]}) r_2^{[1]} , \text{ for } f \in A^*.
\end{align*}
\end{prop}

\begin{proof}
It has been shown in \cite[Example 2.6]{U08} that an skew-symmetric element $r = \sum r^{[1]} \otimes r^{[2]}$ is an associative Yang-Baxter solution if and only if the map $r^\sharp : A^* \rightarrow A$, $r^\sharp (f) = \sum f (r^{[2]}) r^{[1]}$ is an $\mathcal{O}$-operator on $A^*$ over the algebra $A$. Hence the skew-symmetric elements $r_1, r_2, r_1 + r_2$ are associative Yang-Baxter solutions if and only if $r_1^\sharp, r_2^\sharp, r_1^\sharp + r_2^\sharp : A^* \rightarrow A$ are all $\mathcal{O}$-operatots. In other words, $(r_1, r_2)$ is a compatible associative Yang-Baxter solution if and only if $(r_1^\sharp, r_2^\sharp)$ is a compatible $\mathcal{O}$-operator on $A^*$ over the algebra $A$.
\end{proof}

\medskip

\medskip

\section{Cohomology Theory}\label{sec-coho}
Our aim in this section is to introduce cohomology for compatible $\mathcal{O}$-operators and compatible $\mathcal{O}$-operator algebras. We first construct a graded Lie algebra whose Maurer-Cartan elements are compatible $\mathcal{O}$-operators. This characterization of a compatible $\mathcal{O}$-operator $(T_1,T_2)$ gives rise to a cohomology associated to $(T_1, T_2)$. In the second part of this section, we construct an $L_\infty$-algebra whose Maurer-Cartan elements are compatible $\mathcal{O}$-operator algebras. Hence we may also define the cohomology of a compatible $\mathcal{O}$-operator algebra from this Maurer-Cartan characterization.

\medskip

\medskip

\noindent {\bf 3.A. Cohomology of compatible $\mathcal{O}$-operators.} The cohomology of an $\mathcal{O}$-operator (in the associative context) was defined in \cite{D20}. In this section, we introduce the cohomology of a compatible $\mathcal{O}$-operator and finds a relation with the cohomology of the corresponding {total $\mathcal{O}$-operator}. An application of our cohomology will be given in the next section.

\medskip

Let $A$ be an associative algebra and $M$ be an $A$-bimodule. Using the derived bracket \cite{V05}, the author in \cite{U08} constructs a graded Lie bracket $[ ~, ~ ]$ on the graded space $ C^\bullet_\mathrm{O} (M,A) = \oplus_{n \geq 0} C^n_\mathrm{O} (M,A),$ where $C^n_\mathrm{O} (M,A) = \mathrm{Hom}(M^{\otimes n}, A)$. The explicit description of the bracket is given by (see \cite{D20})
\begin{align}\label{dgla-b}
&[ P, Q ] (u_1, \ldots, u_{m+n}) \\
&=  \sum_{i=1}^m (-1)^{(i-1)n}~ P ( u_1, \ldots, u_{i-1}, Q (u_i, \ldots, u_{i+n-1}) \cdot u_{i+n}, \ldots, u_{m+n}) \nonumber \\
&- \sum_{i=1}^m (-1)^{in} ~ P (u_1, \ldots, u_{i-1}, u_i \cdot Q (u_{i+1}, \ldots, u_{i+n}), u_{i+n+1}, \ldots, u_{m+n})  \nonumber \\
&- (-1)^{mn} \bigg\{   \sum_{i=1}^n (-1)^{(i-1)m}~ Q ( u_1, \ldots, u_{i-1}, P (u_i, \ldots, u_{i+m-1}) \cdot u_{i+m}, \ldots, u_{m+n})      \nonumber \\
&- \sum_{i=1}^n (-1)^{im} ~ Q (u_1, \ldots, u_{i-1}, u_i \cdot P (u_{i+1}, \ldots, u_{i+m}), u_{i+m+1}, \ldots, u_{m+n})     \bigg\}  \nonumber  \\
& + (-1)^{mn} \big\{   P(u_1, \ldots, u_m) \cdot Q (u_{m+1}, \ldots, u_{m+n}) - (-1)^{mn} ~ Q (u_1, \ldots, u_n) \cdot P (u_{n+1}, \ldots, u_{m+n}) \big\},  \nonumber 
\end{align}
\begin{align}
&[ P, a ] (u_1, \ldots, u_m) = \sum_{i=1}^m ~ P (u_1, \ldots, u_{i-1}, a \cdot u_i - u_i \cdot a, u_{i+1}, \ldots, u_m) \\
& \qquad \qquad \qquad \qquad \qquad + P (u_1, \ldots, u_m) \cdot a - a \cdot P (u_1, \ldots, u_m), \nonumber 
\end{align}
\begin{align}
[ a, b ] = a \cdot b - b \cdot a, 
\end{align}
for $P \in \mathrm{Hom}(M^{\otimes m},A), Q \in \mathrm{Hom}(M^{\otimes n},A)$ and $a, b \in A$.
The following result has been proved in \cite{D20}.
\begin{prop} A linear map $T : M \rightarrow A$ is an $\mathcal{O}$-operator on $M$ over the algebra $A$ if and only if $T \in \mathrm{Hom}(M,A)$ is a Maurer-Cartan element in the graded Lie algebra $\big( C^\bullet_\mathrm{O} (M,A) , [ ~, ~ ] \big).$
\end{prop}

Let $T : M \rightarrow A$ be an $\mathcal{O}$-operator on $M$ over the algebra $A$. It follows that $T$ induces a differential 
\begin{align*}
\delta_T := [ T, - ] : C^n_\mathrm{O} (M,A) \rightarrow C^{n+1}_\mathrm{O} (M,A), \text{ for } n \geq 0.
\end{align*}
The cohomology groups of the cochain complex $\{ C^\bullet_\mathrm{O}(M,A), \delta_T \}$ are called the cohomology of the $\mathcal{O}$-operator $T$.



Next, we aim to construct a graded Lie algebra whose Maurer-Cartan elements are precisely compatible $\mathcal{O}$-operators. To construct such a graded Lie algebra, we recall a basic construction \cite[Proposition 2.9]{cdm}. 

\begin{prop}\label{new-gla-prop}  Let $(\mathfrak{g} = \oplus_{n \geq 0} \mathfrak{g}^n, [~,~])$ be a graded Lie algebra. Consider the graded vector space $\mathfrak{g}_c = \oplus_{n \geq 0} (\mathfrak{g}_c)^n$, where
\begin{align*}
(\mathfrak{g}_c)^0 = \mathfrak{g}^0 ~~~~ \text{ and } ~~~~ (\mathfrak{g}_c)^n = \underbrace{\mathfrak{g}^n \oplus \cdots \oplus \mathfrak{g}^n}_{(n+1) \mathrm{ ~summand}}.
\end{align*}

(i) Then $(\mathfrak{g}_c, \llbracket ~, ~ \rrbracket)$ is a graded Lie algebra, where
\begin{align}\label{new-bracket}
\llbracket (x_1, \ldots, x_{m+1}), (y_1, \ldots, y_{n+1}) \rrbracket 
 = \big( [x_1, y_1], \ldots, \underbrace{[x_1, y_i]+[x_2, y_{i-1}] + \cdots + [x_i, y_1]}_{i\mathrm{-th~ place}} , [x_{m+1}, y_{m+1}] \big).
\end{align}

(ii) The map $\Theta : \mathfrak{g}_c \rightarrow \mathfrak{g}$ defined by $\Theta (x)  = x$ and $\Theta ((x_1, \ldots, x_{m+1})) = x_1 + \cdots + x_{m+1}$, for $x \in (\mathfrak{g}_c)^0$ and $(x_1, \ldots, x_{m+1}) \in (\mathfrak{g}_c)^{m \geq 1}$, is a morphism of graded Lie algebras.

(iii) Two elements $\theta_1 , \theta_2 \in \mathfrak{g}^1$ satisfies $[\theta_1, \theta_1] = [\theta_2, \theta_2] = [\theta_1, \theta_2] = 0 $ if and only if $(\theta_1, \theta_2)\in (\mathfrak{g}_c)^1$ is a Maurer-Cartan element in the graded Lie algebra $(\mathfrak{g}_c, \llbracket ~, ~ \rrbracket)$.
\end{prop}

In the following, we apply this construction to the graded Lie algebra $\big( C^\bullet_\mathrm{O} (M,A) , [ ~, ~ ]  \big).$ Let $A$ be an associative algebra and $M$ be an $A$-bimodule. Define a graded vector space $C^\bullet_{\mathrm{cO}} (M,A)$ as
\begin{align}\label{new-gla}
C^0_\mathrm{cO}(M,A) = A ~~~~ \text{ and } ~~~~ C^{n \geq 1}_\mathrm{cO}(M,A) = \underbrace{   \mathrm{Hom}(M^{\otimes n}, A) \oplus \cdots \oplus  \mathrm{Hom}(M^{\otimes n}, A)}_{(n+1) \text{ summand}}.
\end{align}
Then $(C^\bullet_\mathrm{cO}(M,A), \llbracket ~, ~ \rrbracket)$ is a graded Lie algebra, where the bracket $\llbracket ~,~ \rrbracket$ is defined by (\ref{new-bracket}).

It follows that a pair $(T_1, T_2)$ of linear maps $T_1, T_2 : M \rightarrow A$ forms a compatible $\mathcal{O}$-operator if and only if $(T_1, T_2) \in C^1_\mathrm{cO} (M,A)$ is a Maurer-Cartan element in the graded Lie algebra $(C^\bullet_\mathrm{cO}(M,A), \llbracket ~, ~ \rrbracket)$. Therefore, a compatible $\mathcal{O}$-operator $(T_1, T_2)$ induces a differential
\begin{align*}
\delta_{(T_1, T_2)} := \llbracket (T_1, T_2), - \rrbracket : C^n_\mathrm{cO}(M,A) \rightarrow C^{n+1}_\mathrm{cO}(M,A), \text{ for } n \geq 0.
\end{align*}
Explicitly, for $a \in C^0_\mathrm{cO}(M,A)$ and $(f_1, \ldots, f_{n+1}) \in C^{n \geq 1}_\mathrm{cO}(M,A),$ we have
\begin{align*}
\delta_{(T_1, T_2)} (a) =~& ([T_1, a], [T_2, a]), \\
\delta_{(T_1, T_2)} ((f_1, \ldots, f_{n+1})) =~& \big( [T_1, f_1], \ldots, \underbrace{[T_1, f_i]+ [T_2, f_{i-1}]}_{i\text{-th place}}, \ldots, [T_2, f_{n+1}]  \big).
\end{align*}

\begin{defn}
The cohomology groups of the cochain complex $\{ C^\bullet_\mathrm{cO}(M,A), \delta_{(T_1, T_2)} \}$ are called the cohomology of the compatible $\mathcal{O}$-operator $(T_1, T_2)$, and they are denoted by $H^\bullet_{(T_1,T_2)} (M,A)$.
\end{defn}

Next, we show that the cohomology of the compatible $\mathcal{O}$-operator $(T_1, T_2)$ can be seen as the cohomology of the induced compatible associative algebra given in Theorem \ref{induced-comp-th}. To show this, we first recall the general cohomology theory of compatible associative algebras. Let $(A, \cdot_1, \cdot_2)$ be a compatible associative algebra and $(M, l_1, r_1, l_2, r_2)$ be a compatible bimodule over it. For each $n \geq 0$, let $C^n_\mathrm{cAss}(A, M)$ be the abelian group defined by
\begin{align*}
C^0_\mathrm{cAss}(A, M) =~& \{ m \in M ~|~ a \cdot_1 m - m \cdot_1 a = a \cdot_2 m - m \cdot_2 a, \text{ for all } a \in A \},\\
C^n_\mathrm{cAss}(A, M) =~& \underbrace{\mathrm{Hom}(A^{\otimes n}, M) \oplus \cdots \oplus \mathrm{Hom}(A^{\otimes n}, M)}_{n\text{ times}}, \text{ for } n \geq 1.
\end{align*}
There is a map $\delta_\mathrm{cAss} : C^n_\mathrm{cAss}(A, M) \rightarrow C^{n+1}_\mathrm{cAss}(A, M)$, for $n \geq 0$, given by
\begin{align*}
\delta_\mathrm{cAss} (m) :=~& \delta^1_\mathrm{Ass} (m) = \delta^2_\mathrm{Ass} (m),\\
\delta_\mathrm{cAss} (f_1, \ldots, f_n) :=~& \big( \delta^1_\mathrm{Ass} (f_1) , \ldots, \underbrace{\delta^1_\mathrm{Ass} (f_i) + \delta^2_\mathrm{Ass}(f_{i-1})}_{i\text{-th place}}, \ldots, \delta^2_\mathrm{Ass} (f_n)   \big),
\end{align*}
for $m \in C^0_\mathrm{cAss}(A, M)$ and $(f_1, \ldots, f_n) \in C^{n \geq 1}_\mathrm{cAss}(A, M)$. Here $\delta^1_\mathrm{Ass}$ is the Hochschild coboundary operator of the associative algebra $(A, \cdot_1)$ with coefficients in the bimodule $(M, l_1, r_1)$, and $\delta^2_\mathrm{Ass}$ is the Hochschild coboundary operator of the associative algebra $(A, \cdot_2)$ with coefficients in the bimodule $(M, l_2, r_2)$. It has been observed in \cite{cdm} that $\{ C^\bullet_\mathrm{cAss}(A, M), \delta_\mathrm{cAss} \}$ is a cochain complex. The corresponding cohomology groups are called the cohomology of the compatible associative algebra $(A, \cdot_1, \cdot_2)$ with coefficients in the compatible bimodule $(M, l_1, r_1, l_2, r_2)$, and they are denoted by $H^\bullet_\mathrm{cAss}(A, M)$.

Next, let $A$ be an associative algebra, $M$ be an $A$-bimodule and $(T_1, T_2)$ be a compatible $\mathcal{O}$-operator on $M$ over the algebra $A$. Consider the induced compatible associative algebra $(M. \star_1, \star_2)$ and its compatible bimodule $(A, l_{T_1}, r_{T_1}, l_{T_2}, r_{T_2})$. Then it can be easily checked that (similar to \cite[Proposition 3.3]{D20})
\begin{align*}
\delta^1_\mathrm{Ass} (f) =  (-1)^n ~[T_1, f]   ~~~ \text{ and } ~~~ \delta^2_\mathrm{Ass} (f) =  (-1)^n ~[T_2, f], \text{ for } f \in \mathrm{Hom}(M^{\otimes n}, A).
\end{align*} 
Here $\delta^1_\mathrm{Ass}$ ~(resp. $\delta^1_\mathrm{Ass}$) is the Hochschild coboundary operator of the associative algebra $(M, \star_1)$ ~(resp. $(M, \star_2)$) with coefficients in the bimodule $(A, l_{T_1}, r_{T_1})$ ~(resp. $(A, l_{T_2}, r_{T_2})$). As a consequence, we get the following.

\begin{thm}
Let $(T_1, T_2)$ be a compatible $\mathcal{O}$-operator on $M$ over the algebra $A$. Then the cohomology of $(T_1, T_2)$ is isomorphic to the cohomology of the induced compatible associative algebra $(M, \star_1, \star_2)$ with coefficients in the compatible bimodule $(A, l_{T_1}, r_{T_1}, l_{T_2}, r_{T_2})$. 
\end{thm}

\medskip

\medskip

Let $(T_1, T_2)$ be a compatible $\mathcal{O}$-operator on $M$ over the algebra $A$. Consider the {total $\mathcal{O}$-operator} $T_1 + T_2$. We define a map
\begin{align*}
\Theta : C^\bullet_\mathrm{cO} (M,A) \rightarrow C^\bullet_\mathrm{O} (M,A), ~~~ \Theta ((f_1, \ldots, f_{n+1} )) = f_1 + \cdots + f_{n+1}, \text{ for } n \geq 0.
\end{align*}
Then it follows from Proposition \ref{new-gla-prop} that the map $\Theta$ is a graded Lie algebra morphism from $( C^\bullet_\mathrm{cO} (M,A), \llbracket ~, ~ \rrbracket )$ to  $( C^\bullet_\mathrm{O} (M,A), [~,~] )$. Moreover, we have $\Theta ((T_1 , T_2)) = T_1 + T_2$. As a consequence, we get the following.

\begin{thm}
Let $(T_1, T_2)$ be a compatible $\mathcal{O}$-operator on $M$ over the algebra $A$. Then $\Theta$ induces a morphism $H^\bullet_{(T_1,T_2)} (M,A) \rightarrow H^\bullet_{T_1+ T_2} (M,A)$ from the cohomology of the compatible $\mathcal{O}$-operator $(T_1, T_2)$ to the cohomology of the total $\mathcal{O}$-operator $T_1+ T_2$.
\end{thm}

\medskip

\medskip

\noindent {\bf 3.B. Cohomology of compatible $\mathcal{O}$-operator algebras.} Here we construct an $L_\infty[1]$-algebra whose Maurer-Cartan elements are precisely compatible $\mathcal{O}$-operator algebras. Using this characterization, we define cohomology of compatible $\mathcal{O}$-operator algebras. We start with some basics of $L_\infty$-algebras.

\begin{defn}
An $L_\infty$-algebra is a pair $(L = \oplus_{i \in \mathbb{Z}} L_i, \{ l_k \}_{k=1}^\infty)$ consisting of a graded vector space $L = \oplus_{i \in \mathbb{Z}} L_i$ together with a collection $\{ l_k \}_{k=1}^\infty$ of degree $1$ multilinear maps $l_k : L^{\otimes k} \rightarrow L$ (for $k \geq 1$) that satisfy

- graded symmetry: ~~$l_k(x_{\sigma(1)},\ldots,x_{\sigma(k)})=\epsilon(\sigma)l_k(x_1,\ldots,x_k),$ for any $\sigma\in S_k$ and $k \geq 1$,

- higher Jacobi identities:
\begin{align*}
\sum_{i+j=n+1}\sum_{\sigma\in S_{(i,n-i)}}\epsilon(\sigma)l_j\left(l_i\left(x_{\sigma(1)},\ldots,x_{\sigma(i)} \right) ,x_{\sigma(i+1)},\ldots,x_{\sigma(n) }\right)=0,
\end{align*}
for $x_1, \ldots, x_n \in L$ and $n \geq 1$. Here $S_{(i,n-i)}$ is the set of all $(i, n-i)$-shuffles and $\epsilon (\sigma) = \epsilon (\sigma; x_1, \ldots, x_n)$ is the standard Koszul sign in the graded context.
\end{defn}

Throughout the paper, all $L_\infty$-algebras are assumed to be weakly filtered \cite{DasMishra}. In other words, certain infinite summations are always convergent. See \cite{getzler} for more details.

\begin{defn}
Let $(L = \oplus_{i \in \mathbb{Z}} L_i, \{ l_k \}_{k=1}^\infty)$  be an $L_\infty$-algebra. An element $\alpha \in L_0$ is said to be a Maurer-Cartan element if it satisfies
\begin{align*}
\sum_{k=1}^\infty \frac{1}{k !} ~\! l_k (\alpha, \ldots, \alpha) = 0.
\end{align*}
\end{defn}

A Maurer-Cartan element $\alpha \in L_0$ induces a degree $1$ differential $l_1^\alpha : L \rightarrow L$ given by
\begin{align*}
l_1^\alpha (x) = \sum_{i=0}^\infty \frac{1}{i !} ~ \! l_{i+1} (\underbrace{\alpha, \ldots, \alpha}_{i \text{ times}}, x).
\end{align*}
In other words, $(L, l_1^\alpha)$ is a cochain complex. The corresponding cohomology groups are called the cohomology induced by the Maurer-Cartan element $\alpha$.

\medskip

An important class of $L_\infty$-algebras arise from $V$-datas \cite{V05}. Recall that a $V$-data is a quadrupe $(V, \mathfrak{a}, P, \triangle)$ in which $V$ is a graded Lie algebra (with the graded Lie bracket $[~,~]$), $\mathfrak{a} \subset V$ is an abelian graded Lie subalgebra, $P : V \rightarrow \mathfrak{a}$ is a projection map with the property that $\mathrm{ker}(P) \subset V$ is a graded Lie subalgebra, and $\triangle \in \mathrm{ker}(P)_1$ that satisfies $[\triangle, \triangle] = 0$. Given a $V$-data, there are few constructions of $L_\infty$-algebras. In the following, we only mention the one which is relevant in our context \cite{V05}.

\begin{thm}\label{v-th}
Let $(V, \mathfrak{a}, P, \triangle)$ be a $V$-data. Suppose $V' \subset V$ is a graded Lie subalgebra that satisfies $[\triangle, V'] \subset V'$. Then the graded vector space $V'[1] \oplus \mathfrak{a}$ carries an $L_\infty$-algebra structure with the multilinear operations
\begin{align*}
l_1 (x[1], a) =~& \big(  - [\triangle, x][1], ~P (x + [\triangle, a]) \big),\\
l_2 (x[1], y[1]) =~& (-1)^{|x|} ~ [x, y][1],\\
l_k (x[1], a_1, \ldots, a_{k-1}) =~& P [ \cdots [[ x, a_1], a_2], \ldots, a_{k-1} ],~ \text{ for } k \geq 2,\\
l_k (a_1, \ldots, a_k) =~& P [\cdots [[ \triangle, a_1], a_2], \ldots, a_k ], ~ \text{ for } k \geq 2.
\end{align*}
Here $x, y$ are homogeneous elements in $V'$ and $a_1, \ldots, a_k$ are homogeneous elements in $\mathfrak{a}$. Up to the permutations of the above entries, all other multilinear operations vanish.
\end{thm}

Using the structure maps of the $L_\infty$-algebra constructed in Theorem \ref{v-th}, we will now construct a new one which will be useful to study compatible $\mathcal{O}$-operator algebras.

\begin{thm}\label{vp-ac}
Let $(V, \mathfrak{a}, P, \triangle)$ be a $V$-data. Suppose $V' \subset V$ is a graded Lie subalgebra that satisfies $[\triangle, V'] \subset V'$. Then the graded vector space $V'[1] \oplus \mathfrak{a}_c$, where
\begin{align*}
(V'[1] \oplus \mathfrak{a}_c)_i = V'_{i+1} \oplus \underbrace{\mathfrak{a}_i \oplus \cdots \oplus \mathfrak{a}_i}_{i+2 ~\mathrm{ times}}, \text{ for } i \geq -1
\end{align*}
carries an $L_\infty$-algebra structure with multilinear operations 
\begin{align*}
&\widetilde{l}_k  \big( \big( a_{1,0}, (a_{1,1}, \ldots, a_{1 ,i_1 + 2}) \big),  \ldots, \big( a_{k,0}, (a_{k,1}, \ldots, a_{k ,i_k + 2})   \big)     \big) \\
&= \bigg(   l_k (a_{1, 0}, \ldots, a_{k, 0}), \sum_{j_1 + \cdots + j_k = 2k-2} l_k (a_{1, j_1}, \ldots, a_{k, j_k}) , \ldots, \sum_{j_1 + \cdots + j_k = i_1 + \cdots + i_k + 2k}  l_k (a_{1, j_1}, \ldots, a_{k, j_k}) \bigg).
\end{align*}
Moreover, the collection of maps 
\begin{align*}
\{ \Theta_i : (V'[1] \oplus \mathfrak{a}_c)_i \rightarrow (V'[1] \oplus \mathfrak{a})_i \}_{i=-1}^\infty \text{ defined by } \Theta_i (x, (a_1, \ldots, a_{i+2})) = (x, a_1 + \cdots + a_{i+2})
\end{align*}
 is a morphism of $L_\infty$-algebras.
\end{thm}

This result generalizes \cite[Proposition 2.9]{cdm} from the context of graded Lie algebras to the context of $L_\infty$-algebras. The $L_\infty$ identities of $V'[1] \oplus \mathfrak{a}_c$ follows from the $L_\infty$ identities of $V'[1] \oplus \mathfrak{a}$. The verification is similar to the above mentioned reference. Hence we will repeat here. By restricting the above structure on the graded vector space $\mathfrak{a}_c$, we obtain the following.

\begin{thm}\label{v-ac}
Let $(V, \mathfrak{a}, P, \triangle)$ be a $V$-data. Suppose $V' \subset V$ is a graded Lie subalgebra that satisfies $[\triangle, V'] \subset V'$. Then the graded vector space $\mathfrak{a}_c = \oplus_{i \geq -1} (\mathfrak{a}_c)_i$, where $(\mathfrak{a}_c)_i = \underbrace{\mathfrak{a}_i \oplus \cdots \oplus \mathfrak{a}_i}_{i+2 ~\mathrm{ times}}$, inherits an $L_\infty$-algebra structure with  multilinear operations
\begin{align*}
&\widetilde{l}_k  \big( (a_{1,1}, \ldots, a_{1 ,i_1 + 2}),  (a_{2,1}, \ldots, a_{2 ,i_2 + 2}) , \ldots,  (a_{k,1}, \ldots, a_{k ,i_k + 2})        \big) \\
&= \bigg( \sum_{ \substack{j_1 + \cdots + j_k \\= 2k-2}}  l_k (a_{1, j_1}, \ldots, a_{k, j_k}), \sum_{ \substack{j_1 + \cdots + j_k \\= 2k-1}}  l_k (a_{1, j_1}, \ldots, a_{k, j_k}) , \ldots, \sum_{ \substack{j_1 + \cdots + j_k \\= i_1 + \cdots + i_k + 2k}}  l_k (a_{1, j_1}, \ldots, a_{k, j_k}) \bigg).
\end{align*}
\end{thm}

\medskip

We will use these results in appropriate setting to study compatible structures. Let $A$ and $M$ be two vectoe spaces. Let $\mathcal{A}^{k, l}$ be the direct sum of all possible $(k+l)$ tensor powers of $A$ and $M$ in which $A$ appears $k$ times (hence $M$ appears $l$ times). For example, 
\begin{align*}
\mathcal{A}^{2,0} = A \otimes A, ~~~ \mathcal{A}^{1,1} = (A \otimes M) \oplus (M \otimes A) ~\text{ and }~ \mathcal{A}^{2,1} = (A \otimes A \otimes M) \oplus (A \otimes M \otimes A) \oplus (M \otimes A \otimes A).
\end{align*}
Note that there is a vector space isomorphism $(A \oplus M)^{\otimes n+1} ~\cong~ \oplus_{k+l =n+1} \mathcal{A}^{k, l}$. Therefore, there is an isomorphism
\begin{align*}
\mathrm{Hom} \big( (A \oplus M)^{\otimes n+1}, A \oplus M  \big) ~ \cong ~ \sum_{k+l = n+1} \mathrm{Hom} (\mathcal{A}^{k,l}, A) \oplus \sum_{k+l =n+1} \mathrm{Hom}(\mathcal{A}^{k,l}, M).
\end{align*}
A multilinear map $f \in \mathrm{Hom} \big( (A \oplus M)^{\otimes n+1}, A \oplus M  \big)$ is said to have bidegree $k | l$ with $k+l =n$, if
\begin{align*}
f (\mathcal{A}^{k+1, l}) \subset A, ~~~ f (\mathcal{A}^{k, l+1}) \subset M ~~ \text{ and } ~~ f \text{ is zero otherwise.}
\end{align*}
A multilinear map is called homogeneous if it has a bidegree. We denote the set of all homogeneous multilinear maps of bidegree $k|l$ by $\mathrm{Hom}^{k|l} (A \oplus M, A \oplus M)$. It is easy to see that
\begin{align*}
\mathrm{Hom}^{k|0} (A \oplus M, A \oplus M) ~ \cong ~& \mathrm{Hom}(A^{\otimes k+1}, A) \oplus \mathrm{Hom}(\mathcal{A}^{k, 1}, M),\\
\mathrm{Hom}^{-1| l} (A \oplus M, A \oplus M) ~ \cong ~& \mathrm{Hom}(M^{\otimes l}, A).
\end{align*}

\medskip

Next, we consider the graded vector space $V := \oplus_{n \geq 0} \mathrm{Hom} \big( (A \oplus M)^{\otimes n+1}, A \oplus M  \big)$ with the Gerstenhaber graded Lie bracket 
\begin{align*}
[f,g]_\mathsf{G} := \sum_{i=1}^{m+1} (-1)^{(i-1) n} f \circ_i g ~-~(-1)^{mn} \sum_{i=1}^{n+1} (-1)^{(i-1) m} g \circ_i f, ~\text{ where }
\end{align*}
\begin{align*}
(f \circ_i g) (v_1, \ldots, v_{m+n+1}) =f \big( v_1, \ldots, v_{i-1}, g(v_i, \ldots, v_{i+n}), \ldots, v_{m+n+1}   \big),
\end{align*}
for $f \in \mathrm{Hom} \big( (A \oplus M)^{\otimes m+1}, A \oplus M  \big)$, $g \in \mathrm{Hom} \big( (A \oplus M)^{\otimes n+1}, A \oplus M  \big)$ and $v_1, \ldots, v_{m+n+1} \in A \oplus M$. Then it has been observed by Uchino \cite{U10} that
\begin{align*}
&V' = \oplus_{k \geq 0} \mathrm{Hom}^{k|0} (A \oplus M, A \oplus M) = \oplus_{k \geq 0} \big(   \mathrm{Hom}(A^{\otimes k+1}, A) \oplus \mathrm{Hom}(\mathcal{A}^{k, 1}, M) \big) \subset V \\
& \text{ is a graded Lie subalgebra,}  \\
&\mathfrak{a} = \oplus_{l \geq 0} \mathrm{Hom}^{-1|l+1} (A \oplus M, A \oplus M) = \oplus_{l \geq 0} \mathrm{Hom}(M^{\otimes l+1}, A) \subset V \text{ is an abelian subalgebra.}
\end{align*}
Let $P: V \rightarrow \mathfrak{a}$ be the projection onto the subspace $\mathfrak{a}$. Then $\mathrm{ker} (P) \subset V$ is a graded Lie subalgebra. With all these notations, the quadruple $(V, \mathfrak{a}, P, \triangle = 0)$ is a $V$-data. Moreover, $V'$ satisfy the assumptions of Theorem \ref{v-th}. Hence one obtains an $L_\infty$-algebra structure on the graded vector space $V'[1] \oplus \mathfrak{a} = \oplus_{i \geq 0} (V'[1] \oplus \mathfrak{a})_i$ with the structure maps
\begin{align*}
l_1 (f[1], \theta) =~& P(f),\\
l_2 (f[1], g[1]) =~& (-1)^{|f|}~ [f, g]_\mathsf{G} [1],\\
l_k (f[1], \theta_1, \ldots, \theta_{k-1}) =~& P [\cdots [[ f , \theta_1]_\mathsf{G}, \theta_2 ]_\mathsf{G}, \ldots, \theta_{k-1}]_\mathsf{G}, ~\text{ for } k \geq 2
\end{align*}
and up to the permutations of the above entries, other multilinear operations vanish.

Next, suppose the vector spaces $A$ and $M$ are equipped with the following maps
\begin{align*}
\mu \in \mathrm{Hom}(A^{\otimes 2}, A), ~~~~ l \in \mathrm{Hom} (A \otimes M , M), ~~~~ r \in \mathrm{Hom}(M \otimes A, M) ~~\text{ and } ~~ T \in \mathrm{Hom}(M,A).
\end{align*}
Consider the element $\pi = \mu + l  + r \in \mathrm{Hom}^{1|0} (A \oplus M, A \oplus M) = V_1'$. Note that $\pi$ can be considered as an element $\pi [1] \in (V'[1])_0.$ Then the following result has been proved in \cite{DasMishra}.

\begin{thm}
With all the above notations, the triple $\big(  A= (A, \mu), M = (M, l, r), T  \big)$ is an $\mathcal{O}$-operator algebra if and only if the element $ (\pi [1], T) \in (V'[1] \oplus \mathfrak{a})_0$ is a Maurer-Cartan element in the $L_\infty$-algebra $(V'[1] \oplus \mathfrak{a}, \{ l_k \}^\infty_{k=1})$.
\end{thm}

Let $\big(  A= (A, \mu), M = (M, l, r), T  \big)$ be a given $\mathcal{O}$-operator algebra. Consider the Maurer-Cartan element $(\pi [1], T) \in (V'[1] \oplus \mathfrak{a})_0$ in the $L_\infty$-algebra $(V'[1] \oplus \mathfrak{a}, \{ l_k \}^\infty_{k=1})$. Thus, we may consider the cochain complex $\{ C^\bullet_\mathrm{OA} (A, M, T), \delta_\mathrm{OA} \}$ induced by the Maurer-Cartan element, where $C^0_\mathrm{OA} (A, M, T) = 0$ and
\begin{align*}
 C^n_\mathrm{OA} (A,M,T) = \begin{cases}  (V'[1])_{-1} = \mathrm{Hom}(A, A) \oplus \mathrm{Hom}(M,M)  & \text{ if } n =1 \\\\
(V'[1] \oplus \mathfrak{a})_{n-2} & \text{ if } n \geq 2 \\
= \mathrm{Hom}(A^{\otimes n}, A) \oplus \mathrm{Hom}(\mathcal{A}^{n-1,1}, M) \oplus \mathrm{Hom}(M^{\otimes n-1}, A) \end{cases}
\end{align*} 
and the coboundary map $\delta_\mathrm{OA} :  C^n_\mathrm{OA} (A,M,T) \rightarrow  C^{n+1}_\mathrm{OA} (A,M,T)$ given by
\begin{align*}
\delta_\mathrm{OA} (f,P) = (-1)^{n-2}~ l_1^{(\pi [1],T)} (f[1], P), \text{ for } f \in V'_{n-1},~ P \in \mathfrak{a}_{n-2}.
\end{align*}
The corresponding cohomology groups are denoted by $H^\bullet_\mathrm{OA} (A, M, T)$ and they are called the cohomology of the $\mathcal{O}$-operator algebra $(A, M, T)$. See \cite{DasMishra} for more details.

\medskip

In the following, we use Theorem \ref{vp-ac} to construct an $L_\infty$-algebra whose Maurer-Cartan elements are compatible $\mathcal{O}$-operator algebras.

\begin{thm}
Let $A$ and $M$ be two vector spaces. Suppose there are maps $\mu \in \mathrm{Hom}(A^{\otimes 2}, A)$, $l \in \mathrm{Hom}(A \otimes M, M)$, $r \in \mathrm{Hom}(M \otimes A, M)$ and $T_1, T_2 \in \mathrm{Hom}(M,A)$. Then $(A = (A, \mu), M = (M, l, r), (T_1, T_2))$ is a compatible $\mathcal{O}$-operator algebra if and only if $(\pi [1], (T_1, T_2)) \in (V'[1] \oplus \mathfrak{a}_c)_0$ is a Maurer-Cartan element in the $L_\infty$-algebra $(V'[1] \oplus \mathfrak{a}_c, \{ \widetilde{l}_k \}_{k=1}^\infty)$.
\end{thm}

\begin{proof}
For the element $(\pi [1], (T_1, T_2)) \in (V'[1] \oplus \mathfrak{a}_c)_0$, we have
\begin{align*}
\widetilde{l}_2 \big( (\pi [1], (T_1, T_2)), (\pi [1], (T_1, T_2))  \big) = - [\pi, \pi ]_\mathsf{G} [1],
\end{align*}
\begin{align*}
\widetilde{l}_3 &\big( (\pi [1], (T_1, T_2)), (\pi [1], (T_1, T_2)), (\pi [1], (T_1, T_2))   \big) \\&= \big( 0, [[\pi, T_1]_\mathsf{G} , T_1 ]_\mathsf{G}, [[\pi, T_1]_\mathsf{G} , T_2 ]_\mathsf{G} + [[\pi, T_2]_\mathsf{G} , T_1 ]_\mathsf{G}, [[\pi, T_2]_\mathsf{G} , T_2 ]_\mathsf{G} \big).
\end{align*}
Note that, for the degree reason, the higher multilinear maps $\widetilde{l}_k$ (for $k \geq 4$) are trivial when applying $k$ times the element $(\pi [1], (T_1, T_2)) \in (V'[1] \oplus \mathfrak{a}_c)_0$ into $\widetilde{l}_k$. Thus,
\begin{align*}
&\sum_{k=1}^\infty \frac{1}{k !} \widetilde{l}_k \big( (\pi [1], (T_1, T_2)), \ldots, (\pi [1], (T_1, T_2))   \big) \\
&= \frac{1}{2} ~ \widetilde{l}_2 \big( (\pi [1], (T_1, T_2)), (\pi [1], (T_1, T_2))  \big) + \frac{1}{6}~ \widetilde{l}_3 \big( (\pi [1], (T_1, T_2)), (\pi [1], (T_1, T_2)), (\pi [1], (T_1, T_2))   \big)  \\
&= \big( - \frac{1}{2} [\pi, \pi ]_\mathsf{G} [1], ~\frac{1}{6} [[\pi, T_1]_\mathsf{G} , T_1 ]_\mathsf{G}, ~\frac{1}{3} [[\pi, T_1]_\mathsf{G} , T_2 ]_\mathsf{G} , ~\frac{1}{6} [[\pi, T_2]_\mathsf{G} , T_2 ]_\mathsf{G}   \big).
\end{align*}
This is zero if and only if all the components are zero. Hence $(\pi [1], (T_1, T_2))$ is a Maurer-Cartan element if and only if $[\pi, \pi ]_\mathsf{G} = 0$, $[[\pi, T_1]_\mathsf{G} , T_1 ]_\mathsf{G} = 0$, $[[\pi, T_1]_\mathsf{G} , T_2 ]_\mathsf{G} = 0$ and $[[\pi, T_2]_\mathsf{G} , T_2 ]_\mathsf{G} = 0$.

On the other hand, the triple $(A = (A, \mu), M= (M, l, r), (T_1, T_2))$ is a compatible $\mathcal{O}$-operator algebra if and only if $A= (A, \mu)$ is an associative algebra, $M = (M, l, r)$ is an $A$-bimodule and $(T_1, T_2)$ is a compatible $\mathcal{O}$-operator on $M$ over the algebra $A$. Note that $(A, \mu)$ is an associative algebra and $(M, l, r)$ is an $A$-bimodule if and only if the element $\pi = \mu + l + r$ satisfies $[\pi, \pi]_G = 0$ (see \cite{D20}). Finally, the pair $(T_1, T_2)$ is a compatible $\mathcal{O}$-operator on $M$ over the algebra $A$ if and only if $\llbracket (T_1, T_2), (T_1, T_2) \rrbracket = 0$, or equivalently, $[T_1, T_1] = 0$, $[T_1, T_2] = 0$ and $[T_2, T_2] = 0$. In view of the Gerstenhaber bracket (see \cite{D20}), these conditions are respectively equivalent to $[[\pi, T_1]_\mathsf{G} , T_1 ]_\mathsf{G} = 0$, $[[\pi, T_1]_\mathsf{G} , T_2 ]_\mathsf{G} = 0$ and $[[\pi, T_2]_\mathsf{G} , T_2 ]_\mathsf{G} = 0$. Hence the result follows.
\end{proof}

Let $(A = (A, \mu), M = (M, l, r), (T_1, T_2))$ be a compatible $\mathcal{O}$-operator algebra. It follows from the above theorem that the element $ (\pi [1], (T_1, T_2)) \in (V'[1] \oplus \mathfrak{a}_c)_0$ is a Maurer-Cartan element in the $L_\infty$-algebra $(V'[1] \oplus \mathfrak{a}_c, \{ l_k \}_{k=1}^\infty)$. Consider the corresponding cochain complex induced by the Maurer-Cartan element $ (\pi [1], (T_1, T_2))$. More precisely, we consider the complex $\{ C^\bullet_{\mathrm{cOA}} (A, M, (T_1, T_2)) , \delta_\mathrm{cOA} \}$, where $C^0_{\mathrm{cOA}} (A, M, (T_1, T_2)) = 0$ and
\begin{align*}
C^n_{\mathrm{cOA}} (A, M, (T_1, T_2)) = \begin{cases}  (V'[1])_{-1} = \mathrm{Hom}(A, A) \oplus \mathrm{Hom}(M,M)  \qquad \qquad \text{ if } n =1 \\\\
(V'[1] \oplus \mathfrak{a}_c)_{n-2}  \qquad \qquad \qquad \text{ if } n \geq 2 \\
= \mathrm{Hom}(A^{\otimes n}, A) \oplus \mathrm{Hom}(\mathcal{A}^{n-1,1}, M) \oplus \underbrace{\mathrm{Hom}(M^{\otimes n-1}, A) \oplus \cdots \oplus \mathrm{Hom}(M^{\otimes n-1}, A)}_{n \text{ times}} \end{cases}
\end{align*}
and the coboundary map $\delta_\mathrm{cOA} : C^n_{\mathrm{cOA}} (A, M, (T_1, T_2)) \rightarrow C^{n+1}_{\mathrm{cOA}} (A, M, (T_1, T_2))$ given by
\begin{align*}
\delta_\mathrm{cOA} \big( f, (P_1, \ldots, P_n) \big) := (-1)^{n-2} ~ l_1^{(\pi[1], (T_1, T_2))} \big( f[1], (P_1, \ldots, P_n) \big), \text{ for } f \in V'_{n-1},~ (P_1, \ldots, P_n) \in (\mathfrak{a}_c)_{n-2}.
\end{align*}

\medskip

Let $Z^n_{\mathrm{cOA}} (A, M, (T_1, T_2))$ be the space of all $n$-cocycles and $B^n_{\mathrm{cOA}} (A, M, (T_1, T_2))$ be the space of all $n$-coboundaries. Then we have $B^n_{\mathrm{cOA}} (A, M, (T_1, T_2)) \subset Z^n_{\mathrm{cOA}} (A, M, (T_1, T_2))$, for all $n \geq 0$. The corresponding quotient groups
\begin{align*}
H^n_{\mathrm{cOA}} (A, M, (T_1, T_2)) := \frac{Z^n_{\mathrm{cOA}} (A, M, (T_1, T_2))}{B^n_{\mathrm{cOA}} (A, M, (T_1, T_2))}, ~\text{ for } n \geq 0
\end{align*}
are called the cohomology of the compatible $\mathcal{O}$-operator algebra $(A,M, (T_1, T_2)).$

\medskip

Let $(A,M, (T_1, T_2))$ be a compatible $\mathcal{O}$-operator algebra. Consider the Maurer-Cartan element $\alpha = (\pi [1], (T_1, T_2)) \in (V'[1] \oplus \mathfrak{a}_c)_0$ in the $L_\infty$-algebra $V'[1] \oplus \mathfrak{a}_c$. On the other hand, the triple $(A, M, T_1 + T_2)$ is an $\mathcal{O}$-operator algebra. Hence there is a Maurer-Cartan element $\alpha^\mathrm{Tot} = (\pi [1], T_1 + T_2) \in (V'[1] \oplus \mathfrak{a})_0$ in the $L_\infty$-algebra $V'[1] \oplus \mathfrak{a}$. Therefore, if we consider the $L_\infty$-algebra morphism 
\begin{align*}
\{ \Theta_i : (V'[1] \oplus \mathfrak{a}_c)_i \rightarrow (V'[1] \oplus \mathfrak{a})_i \}_{i=0}^\infty
\end{align*}
given in Theorem \ref{vp-ac}, then we have $\Theta_0 (\alpha) = \alpha^\mathrm{Tot}$. As a consequence, we obtain a morphism of cochain complexes from $\{ C^\bullet_\mathrm{cOA} (A, M, (T_1, T_2)) , \delta_\mathrm{cOA} \}$ to $\{ C^\bullet_\mathrm{OA} (A, M, T_1 + T_2) , \delta_\mathrm{OA} \}$. Hence we get the following result.

\begin{thm}
Let $(A,M, (T_1, T_2))$ be a compatibel $\mathcal{O}$-operator algebra. Then there is a morphism of groups $H^\bullet_\mathrm{cOA} (A, M, (T_1, T_2)) \rightarrow H^\bullet_\mathrm{OA} (A, M, T_1 + T_2)$ from the cohomology of the compatible $\mathcal{O}$-operator algebra $(A,M, (T_1, T_2))$ to the cohomology of the $\mathcal{O}$-operator $(A, M, T_1 + T_2)$.
\end{thm}

\medskip

\noindent {\bf 3.C. Compatible homotopy $\mathcal{O}$-operators.} The notion of homotopy $\mathcal{O}$-operators was introduced in \cite{DasMishra} as the homotopy analogue of $\mathcal{O}$-operators. In this section, we define compatible homotopy $\mathcal{O}$-operators as the homotopy analogue of compatible $\mathcal{O}$-operators. A compatible homotopy $\mathcal{O}$-operator is a pair consisting of two homotopy $\mathcal{O}$-operators satisfying suitable compatibilities. 

We first recall some definitions and basic results about $A_\infty$-algebras and homotopy bimodule over them.

\begin{defn}
An $A_\infty$-algebra is a pair $(A = \oplus_{i \in \mathbb{Z}} A_i, \{ \mu_k \}_{k=1}^\infty)$ consisting of a graded vector space $A = \oplus_{i \in \mathbb{Z}} A_i$ with a collection $\{ \mu_k \}_{k=1}^\infty$ of degree $1$ multilinear maps $\mu_k : A^{\otimes k} \rightarrow A$ (for $k \geq 1$) satisfying the following set of identities
\begin{align}\label{a-inf-iden}
\sum_{i+j = n+1} \sum_{\lambda =1}^j (-1)^{|a_1| + \cdots + |a_{\lambda -1}|}~ \mu_j \big(    a_1, \ldots, a_{\lambda -1} , \mu_i (a_\lambda, \ldots, a_{\lambda+i-1}), a_{\lambda +i}, \ldots, a_n \big) =0,
\end{align}
for $a_1, \ldots, a_n \in A$ and $n \geq 1$.
\end{defn}

Let $(A = \oplus_{i \in \mathbb{Z}} A_i, \{ \mu_k \}_{k=1}^\infty)$ be an $A_\infty$-algebra. A homotopy bimodule over this $A_\infty$-algebra consists of a pair $(M = \oplus_{i \in \mathbb{Z}} M_i, \{ \eta_k \}_{k=1}^\infty)$ of a graded vector space $M = \oplus_{i \in \mathbb{Z}} M_i$ and a collection $\{ \eta_k \}_{k=1}^\infty$ of degree $1$ multilinear maps $\eta_k : \mathcal{A}^{k-1,1} \rightarrow M$ (for $k \geq 1$) satisfying the set of identities  (\ref{a-inf-iden}) with exactly one of the variables $a_1, \ldots, a_n$ comes from $M$ and the  corresponding multilinear map $\mu_i$ or $\mu_j$ replaced by $\eta_i$ or $\eta_j$. Like ungraded case, here $\mathcal{A}^{k-1, 1}$ denotes the direct sum of all possible $k$ tensor powers of $A$ and $M$, in which $A$ appears $k-1$ times (hence $M$ appears exactly once).

Given an $A_\infty$-algebra and a homotopy bimodule over it, we construct an $L_\infty$-algebra following Theorem \ref{v-ac}. This new $L_\infty$-algebra is suitable to study compatible homotopy $\mathcal{O}$-operators. Let $(A, \{ \mu_k \}_{k=1}^\infty)$ be an $A_\infty$-algebra and $(M, \{ \eta_k \}_{k=1}^\infty)$ be a homotopy bimodule over it. We consider the graded Lie algebra 
\begin{align*}
V = \big(  \oplus_{n \in \mathbb{Z}} \mathrm{Hom}^n (\overline{T}(A \oplus M), A \oplus M), [~,~]_\mathsf{G}  \big),
\end{align*}
where an element $f \in \mathrm{Hom}^n (\overline{T}(A \oplus M), A \oplus M)$ is given by a formal summation $f = \sum_{k \geq 1} f_k$ with $f_k \in \mathrm{Hom}^n ((A \oplus M)^{\otimes k} , A \oplus M )$ is a $k$-ary multilinear map of degree $n$. The bracket $[~,~]_\mathsf{G}$ is given by
\begin{align*}
[f_k , g_l ]_\mathsf{G} (v_1, \ldots, v_{k+l-1}) = \sum_{i=1}^k (-1)^{n (|v_1| + \cdots + |v_{i-1}|)} ~ f_k \big( v_1, \ldots, v_{i-1}, g_l (v_i, \ldots, v_{i+l-1}), v_{i+l}, \ldots, v_{k+l-1}    \big)\\
- (-1)^{mn} \sum_{i=1}^l (-1)^{m (|v_1| + \cdots + |v_{i-1}|)} ~ g_l \big( v_1, \ldots, v_{i-1}, f_k (v_i, \ldots, v_{i+k-1}), v_{i+k}, \ldots, v_{k+l-1}    \big),
\end{align*}
for $f_k \in \mathrm{Hom}^m ((A \oplus M)^{\otimes k}, A \oplus M)$, $g_l \in \mathrm{Hom}^n ((A \oplus M)^{\otimes l}, A \oplus M)$ and $v_1, \ldots, v_{k+l-1} \in A \oplus M$. Then we can easily observe that $\mathfrak{a} = \oplus_{n \in \mathbb{Z}} \mathrm{Hom}^n (\overline{T}(M), A)$ is an abelian subalgebra of $V$. Let $P : V \rightarrow \mathfrak{a}$ be the projection onto the subspace $\mathfrak{a}$. We also consider the element $\triangle = \sum_{k \geq 1} (\mu_k + \eta_k) \in \mathrm{ker} (P)_1$. Note that the $A_\infty$-algebra identities of $(A, \{ \mu_k \}_{k=1}^\infty)$ and the homotopy bimodule conditions of $(M, \{ \eta_k \}_{k=1}^\infty)$ is simply equivalent to the condition $[\triangle, \triangle]_\mathsf{G} = 0$. As a summary, we obtain a $V$-data $(V, \mathfrak{a}, P, \triangle)$. Hence by Theorem \ref{v-ac}, we get an $L_\infty$-algebra structure on $\mathfrak{a}_c$. In the ungraded case, this $L_\infty$-algebra turns out to be the graded Lie algebra given in (\ref{new-gla}). This generality allows one to define compatible homotopy $\mathcal{O}$-operator as a Maurer-Cartan element in the $L_\infty$-algebra $\mathfrak{a}_c$. We will study some properties of compatible homotopy $\mathcal{O}$-operators and their cohomological properties in a future project.

\medskip

\medskip

\section{Deformation Theory}\label{sec-def}
Formal deformation theory began with the seminal work of Gerstenhaber for associative algebras \cite{G64}. In this section, we first consider formal deformations of a compatible $\mathcal{O}$-operator $(T_1, T_2)$. In such deformations, we only deform the $\mathcal{O}$-operators $T_1, T_2$ that are compatible and keep the underlying algebra and bimodule intact. Such deformations are deformations are governed by the cohomology of the compatible $\mathcal{O}$-operator. Next, we also consider formal deformations of a compatible $\mathcal{O}$-operator algebra $(A, M, (T_1, T_2)).$ Here we allow deformations of all the structures. Such deformations are governed by the cohomology of the compatible $\mathcal{O}$-operator algebra $(A, M, (T_1, T_2)).$

\medskip

\medskip

\noindent {\bf 4.A. Deformations of compatible $\mathcal{O}$-operators.} Let $A$ be an associative algebra and $M$ be an $A$-bimodule. Consider the space $A[[t]]$ of formal power series in $t$ with coefficients from $A$. Then $A[[t]]$ is an associative algebra over the ring ${\bf k}[[t]]$. Moreover, the $A$-bimodule structure on $M$ can be extended to an $A[[t]]$-bimodule structure on $M[[t]]$.

\begin{defn}
Let $(T_1, T_2)$ be a compatible $\mathcal{O}$-operator on $M$ over the algebra $A$. A formal one-parameter deformation of $(T_1, T_2)$ consists of a pair $(T_{1,t}, T_{2,t})$ of formal sums
\begin{align*}
T_{1,t} = \sum_{i=0}^\infty t^i T_{1, i} ~ \text{ and } ~ T_{2,t} = \sum_{i=0}^\infty t^i T_{2, i}, ~ \text{ where } ~ T_{1,i}, T_{2,i} \in \mathrm{Hom}(M,A) \text{ with } T_{1,0} = T_1, T_{2,0} = T_2
\end{align*}
such that $(T_{1,t}, T_{2,t})$ is a compatible $\mathcal{O}$-operator on $M[[t]]$ over the algebra $A[[t]]$.
\end{defn}

It follows that $(T_{1,t}, T_{2,t})$ is a formal one-parameter deformation if and only if
\begin{align}
\sum_{i+j = n} T_{1,i} (u) \cdot T_{1,j} (v) =~& \sum_{i+j = n} T_{1,i} \big(   T_{1,j} (u) \cdot v + u \cdot T_{1,j} (v) \big), \label{formal1} \\
\sum_{i+j = n} T_{2,i} (u) \cdot T_{2,j} (v) =~& \sum_{i+j = n} T_{2,i} \big(   T_{2,j} (u) \cdot v + u \cdot T_{2,j} (v) \big), \\
\sum_{i+j = n} \big( T_{1,i} (u) \cdot T_{2,j} (v)  + T_{2,i} (u) \cdot T_{1,j} (v)\big) =~& \sum_{i+j = n} T_{1,i} \big(   T_{2,j} (u) \cdot v + u \cdot T_{2,j} (v) \big) + T_{2,i} \big(   T_{1,j} (u) \cdot v + u \cdot T_{1,j} (v) \big) \label{formal3}
\end{align}
holds, for $n = 0, 1, \ldots$. These system of equations are called deformation equations. Note that these equations are hold for $n=0$ as the pair $(T_1, T_2)$ is a compatible $\mathcal{O}$-operator. For $n =1$, it amounts that
\begin{align}
&T_{1}(u) \cdot T_{1,1} (v) + T_{1,1} (u) \cdot  T_{1}(v) = T_{1}\big( u \cdot T_{1,1} (v) + T_{1,1} (u) \cdot v \big)+ T_{1,1} \big(u \cdot T_{1}(v) + T_{1}(u) \cdot v \big), \label{com-def1}\\
&T_{2}(u) \cdot T_{2,1} (v) + T_{2,1} (u) \cdot  T_{2}(v) = T_{2}\big( u \cdot T_{2,1} (v) + T_{2,1} (u) \cdot v \big)+ T_{2,1} \big(u \cdot T_{2}(v) + T_{2}(u) \cdot v \big), \\
&T_{1}(u) \cdot T_{2,1} (v) + T_{2}(u) \cdot T_{1,1} (v) + T_{1,1} (u) \cdot  T_{2}(v)  + T_{2,1} (u) \cdot  T_{1}(v) = T_{1}\big( u \cdot T_{2,1} (v) + T_{2,1} (u) \cdot v \big) \label{com-def3}\\
& \qquad \qquad + T_{2}\big( u \cdot T_{1,1} (v) + T_{1,1} (u) \cdot v \big) + T_{1,1} \big(u \cdot T_{2}(v) + T_{2}(u) \cdot v \big) + T_{2,1} \big(u \cdot T_{1}(v) + T_{1}(u) \cdot v \big), \nonumber
\end{align}
for $u, v \in M$. The first identity is equivalent to $[T_1, T_{1,1}] = 0$, while the second identity is equivalent to $[T_2, T_{2,1}] = 0$. Finally, the third identity is equivalent to $[T_1, T_{2,1}]+ [T_2, T_{1,1}] = 0$. Therefore, from (\ref{com-def1})-(\ref{com-def3}), we get that
\begin{align*}
\delta_{(T_1, T_2)} \big( (T_{1,1},T_{2,1}) \big)= \big( [T_1, T_{1,1}], [T_1, T_{2,1}]+ [T_2, T_{1,1}], [T_2, T_{2,1}]  \big) = 0.
\end{align*}
In other words, $(T_{1,1}, T_{2,1})$ is a $1$-cocycle in the cochain complex $\{ C^\bullet_\mathrm{cO}(M,A), \delta_{(T_1,T_2)} \}$. This is called the `infinitesimal' of the deformation $(T_{1,t}, T_{2,t}).$

\begin{defn}
Let $(T_{1,t}, T_{2,t})$ and $(T'_{1,t}, T'_{2,t})$ be two formal one-parameter deformations of the compatible $\mathcal{O}$-operator $(T_1, T_2)$. They are said to be equivalent if there exists an element $a_0 \in A$, and maps $\phi_i \in \mathrm{Hom}(A,A)$, $\psi_i \in \mathrm{Hom}(M,M)$ for $i \geq 2$, such that the pair
\begin{align*}
\big( \phi_t := \mathrm{id}_A + t (\mathrm{ad}^l_{a_0} - \mathrm{ad}^r_{a_0}) + \sum_{i \geq 2} t^i \phi_i , ~ \psi_t := \mathrm{id}_M + t (l_{a_0} - r_{a_0}) + \sum_{i \geq 2} t^i \psi_i  \big)
\end{align*}
is a morphism of compatible $\mathcal{O}$-operators from $(T_{1,t}, T_{2,t})$ to $(T'_{1,t}, T'_{2,t})$.
\end{defn}

Therefore, the following identities must hold for equivalence of deformations
\begin{align}
&\phi_t (a \cdot b) = \phi_t (a) \cdot \phi_t (b),\\
&(\phi_t \circ T_{1,t})(u) = (T'_{1,t} \circ \psi_t)(u), ~~~~ (\phi_t \circ T_{2,t})(v) = (T'_{2,t} \circ \psi_t)(v), \label{def-eq2}\\
&\psi_t (a \cdot u) = \phi_t (a) \cdot \psi_t (u), ~~~~ \psi_t (u \cdot a) = \psi_t (u) \cdot \phi_t (a), ~ \text{ for } a, b \in A, u , v\in M.
\end{align}
By equating coefficients of $t$ from both sides of the two identities in (\ref{def-eq2}), we get
\begin{align*}
&\big( (T_{1,1}, T_{2,1}) - (T'_{1,1}, T'_{2,1})\big)(u,v)  \\
&= \big(  T_1 (u) \cdot a_0 - T_1 (u \cdot a_0) - a_0 \cdot T_1 (u) + T_1 (a \cdot u), ~  T_2 (v) \cdot a_0 - T_2 (v \cdot a_0) - a_0 \cdot T_2 (v) + T_2 (a \cdot v) \big) \\
&= \big( \delta_{(T_1,T_2)} (a_0) \big) (u, v).
\end{align*}
As a summary of the above discussions, we get the following.

\begin{thm}
Let $(T_1, T_2)$ be a compatible $\mathcal{O}$-operator on $M$ over the algebra $A$. Then the infinitesimal in any formal deformation of $(T_1,T_2)$ is a $1$-cocycle in the cohomology complex of $(T_1,T_2)$. Moreover, the corresponding cohomology class depends only on the equivalence class of the deformation.
\end{thm}

\medskip

\noindent {\bf Rigidity of compatible $\mathcal{O}$-operators.} In the next, we consider rigidity of compatible $\mathcal{O}$-operators following the classical concept of rigidity of associative structures \cite{G64}.

\begin{defn}
A compatible $\mathcal{O}$-operator $(T_1, T_2)$ is said to be rigid if any formal one-parameter deformation $(T_{1,t}, T_{2,t})$ of the compatible $\mathcal{O}$-operator $(T_1,T_2)$ is equivalent to the undeformed one.
\end{defn}

Recall that Nijenhuis elements associated with an $\mathcal{O}$-operator was defined in \cite{D20}. The author also finds a sufficient condition for the rigidity of an $\mathcal{O}$-operator in terms of Nijenhuis elements. Our aim here is to generalize these in the compatible set-up.

\begin{defn}
Let $(T_1, T_2)$ be a compatible $\mathcal{O}$-operator on $M$ over the algebra $A$. An element $a_0 \in A$ is said to be a Nijenhuis element associated with $(T_1, T_2)$ if the element $a_0$ is a Nijenhuis element for both the $\mathcal{O}$-operators $T_1, T_2$, i.e., the followings are hold
\begin{align*}
(a_0 \cdot a - a \cdot a_0) ~\cdot~ &(a_0 \cdot b - b \cdot a_0) = 0,\\
a_0 \cdot \big( l_{T_1} (u , a_0) - r_{T_1} (a_0, u)  \big) ~-~ & \big( l_{T_1} (u , a_0) - r_{T_1} (a_0, u)  \big) \cdot a_0 = 0,\\
a_0 \cdot \big( l_{T_2} (u , a_0) - r_{T_2} (a_0, u)  \big) ~-~ & \big( l_{T_2} (u , a_0) - r_{T_2} (a_0, u)  \big) \cdot a_0 = 0,\\
(a_0 \cdot a - a \cdot a_0) ~\cdot~ & (a \cdot u - u \cdot a) = 0,  \\
(a \cdot u - u \cdot a) ~\cdot~ &(a_0 \cdot a - a \cdot a_0) = 0,
\end{align*}
for $a, b \in A$ and $u \in M$. We denote the set of all Nijenhuis elements by $\mathrm{Nij}(T_1, T_2)$.
\end{defn}

\begin{remark}
It is easy to observe that, if $(T_1, T_2)$ is a compatible $\mathcal{O}$-operator and $a_0 \in \mathrm{Nij}(T_1, T_2)$ then $a_0 \in \mathrm{Nij}(T_1 + T_2)$, where $T_1 + T_2$ is the total $\mathcal{O}$-operator.
\end{remark}

The following result is a generalization of \cite[Proposition 4.16]{D20} in the compatible set-up. Hence we omit the proof.

\begin{prop}
Let $(T_1, T_2)$ be a compatible $\mathcal{O}$-operator on $M$ over the algebra $A$. If any $1$-cocycle in the cohomology complex $\{ C^\bullet_\mathrm{cO} (M,A), \delta_{(T_1,T_2)} \}$ is given by $\delta_{(T_1, T_2)} (a_0)$, for some $a_0 \in \mathrm{Nij}(T_1, T_2)$, then the compatible $\mathcal{O}$-operator $(T_1, T_2)$ is rigid.
\end{prop}
\medskip

\noindent {\bf Finite order deformations.} Let $A$ be an associative algebra and $M$ be an $A$-bimodule. Let $(T_1, T_2)$ be a compatible $\mathcal{O}$-operator on $M$ over the algebra $A$.

\begin{defn}
An order $N$ deformation ($N \in \mathbb{N}$) of the compatible $\mathcal{O}$-operator $(T_1, T_2)$ consists of a pair $(T_{1,t}^N, T_{2,t}^N)$ of two degree $N$ polynomials of the form
\begin{align*}
T_{1,t}^N = \sum_{i=0}^N t^i T_{1, i} ~~~ \text{ and } ~~~ T_{2,t}^N = \sum_{i=0}^N t^i T_{2, i}, ~ \text{ where } ~ T_{1,i}, T_{2,i} \in \mathrm{Hom}(M,A) \text{ with } T_{1,0} = T_1,~ T_{2,0} = T_2
\end{align*}
that satisfies the identities (\ref{formal1})-(\ref{formal3}), for $n =0, 1, \ldots, N$.
\end{defn}

The system of identities can be equivalently expressed as
\begin{align*}
[T_1, T_{1, n}] = - &\frac{1}{2} \sum_{\substack{i+j = n \\ i, j \geq 1}} [T_{1,i}, T_{1,j}], \qquad
[T_2, T_{2, n}] = - \frac{1}{2} \sum_{\substack{i+j = n \\ i, j \geq 1}} [T_{2,i}, T_{2,j}], \\
[T_1, T_{2,n}]&+ [T_2,T_{1,n}] = - \frac{1}{2} \sum_{ \substack{ i+j = n \\ i, j \geq 1}} \big( [T_{1,i}, T_{2,j}] + [T_{2,i}, T_{1,j} ]   \big),
\end{align*}
for $n = 0, 1, \ldots , n$. We may combine these identities and write as
\begin{align}\label{half}
\delta_{(T_1, T_2)} (T_{1,n}, T_{2,n}) = - \frac{1}{2} \sum_{\substack{i+j = n \\ i, j \geq 1}} \llbracket (T_{1, i}, T_{2,i}), (T_{1, j}, T_{2,j}) \rrbracket, ~ \text{ for } n = 0, 1, \ldots, N.
\end{align}
Motivated by the right-hand side expression of (\ref{half}), we define an element
\begin{align*}
\mathrm{Ob}_{ (T_{1,t}^N, T_{2,t}^N)} := - \frac{1}{2} \sum_{\substack{i+j = N + 1 \\ i, j \geq 1}} \llbracket (T_{1, i}, T_{2,i}), (T_{1, j}, T_{2,j}) \rrbracket.
\end{align*}
Note that this is a $2$-cochain in the cochain complex $\{ C^\bullet_\mathrm{cO} (M,A), \delta_{(T_1, T_2)} \}$ defining the cohomology of the compatible $\mathcal{O}$-operator $(T_1, T_2)$. Using the graded Jacobi identity of the bracket $\llbracket ~, ~ \rrbracket$ and the identities (\ref{half}), it can be shown that $\delta_{(T_1, T_2)} \big( \mathrm{Ob}_{ (T_{1,t}^N, T_{2,t}^N)}  \big) = 0.$ See \cite{D20} for a similar calculation. In other words, $\mathrm{Ob}_{ (T_{1,t}^N, T_{2,t}^N)}$ is a $2$-cocycle. The corresponding cohomology class $[\mathrm{Ob}_{ (T_{1,t}^N, T_{2,t}^N)}] \in {H^2_{(T_1, T_2)} (M,A)}$ is called the `obstruction class'.

\begin{defn}
An order $N$ deformation $(T_{1,t}^N, T_{2,t}^N)$ of the compatible $\mathcal{O}$-operator $(T_1, T_2)$ is said to be extensible if there are maps $T_{1, N+1}, T_{2, N+1} : M \rightarrow A$ which makes the pair
\begin{align}\label{ext-def}
\big(  T_{1, t}^{N+1} = T_{1, t}^N + t^{N+1} T_{1, N+1}, ~ T_{2, t}^{N+1} = T_{2, t}^N + t^{N+1} T_{2, N+1} \big)
\end{align}
into an order $N+1$ deformation.
\end{defn}

\begin{thm}
An order $N$ deformation $(T_{1,t}^N, T_{2,t}^N)$ of the compatible $\mathcal{O}$-operator $(T_1, T_2)$ is extensible if and only if the corresponding obstruction class $[\mathrm{Ob}_{ (T_{1,t}^N, T_{2,t}^N)}]$ vanishes.
\end{thm}

\begin{proof}
Suppose $(T_{1,t}^N, T_{2,t}^N)$ is extensible. Then from the definition of $\mathrm{Ob}_{ (T_{1,t}^N, T_{2,t}^N)}$, we have $\mathrm{Ob}_{ (T_{1,t}^N, T_{2,t}^N)} = \delta_{(T_1,T_2)} ((T_{1, N+1}, T_{2, N+1}))$ a coboundary. Hence the corresponding cohomology class vanishes.

Conversely, let $(T_{1,t}^N, T_{2,t}^N)$ be a deformation of order $N$ for which the obstruction class $[\mathrm{Ob}_{ (T_{1,t}^N, T_{2,t}^N)}]$ vanishes. Then we have $ \mathrm{Ob}_{ (T_{1,t}^N, T_{2,t}^N)} = \delta_{(T_1,T_2)} ((T_{1, N+1}, T_{2, N+1}))$, for some $(T_{1, N+1}, T_{2, N+1}) \in C^1_\mathrm{cO} (M,A)$. This implies that the pair $(T^{N+1}_{1,t}, T^{N+1}_{2,t})$ given by (\ref{ext-def}) is a deformation of order $N+1$. In other words, $(T_{1,t}^N, T_{2,t}^N)$ is extensible.
\end{proof}

\medskip

\medskip

\noindent {\bf 4.B. Deformations of compatible $\mathcal{O}$-operator algebras.}
Here we will consider formal deformations of compatible $\mathcal{O}$-operator algebras in which we allow to deform all the underlying structures. 

\begin{defn}
Let $(A = (A, \mu), M = (M, l, r), (T_1, T_2))$ be a compatible $\mathcal{O}$-operator algebra. A 
formal deformation of this compatible $\mathcal{O}$-operator algebra consist of formal sums
\begin{align*}
     \mu_t= \sum_{i=0}^\infty t^i\mu_i ,~~~~   l_t=\sum_{i=0}^\infty t^i l_i ,~~~~ r_t=\sum_{i=0}^\infty t^i r_i,~~~~T_{1,t} = \sum_{i=0}^\infty t^i T_{1, i} ~~ \text{ and } ~~ T_{2,t} = \sum_{i=0}^\infty t^i T_{2, i},
\end{align*}
where $\mu_i\in \mathrm{Hom}(A^{\otimes 2},A)$, $l_i\in \mathrm{Hom}(A\otimes M,A)$, $r_i\in \mathrm{Hom}(M\otimes A,M)$, $T_{1,i}, T_{2,i} \in \mathrm{Hom}(M,A)$ with  $\mu_0=\mu, l_0=l, r_0=r, T_{1,0} = T_1$ and $T_{2,0} = T_2$, such that the triple 
\begin{align*}
\big( A[[t]]= (A[[t]],\mu_t), M[[t]] = (M[[t]],l_t,r_t),(T_{1,t},T_{2,t}) \big)
\end{align*}
is a compatible $\mathcal{O}$-operator algebra. We denote a formal deformation as above by the tuple $(\mu_t, l_t, r_t, T_{1, t}, T_{2,t}).$
\end{defn}

It follows that $( \mu_t, l_t, r_t, T_{1, t}, T_{2,t})$ is a formal deformation of the compatible $\mathcal{O}$-operator algebra if the following set of equations are hold:
\begin{align}
\sum_{i+j= n} \mu_i(\mu_j(a,b),c)=~&\sum_{i+j= n}\mu_i(a,\mu_j(b,c)),\label{formal1} \\
	\sum_{i+j= n} l_i(\mu_j(a,b),u)=~&\sum_{i+j= n}l_i(a,l_j(b,u)),\label{formal2} \\
	\sum_{i+j= n} r_i(l_j(a,u),b)=~&\sum_{i+j= n}l_i(a,r_j(u,b)),\label{formal2p} \\
	\sum_{i+j= n} r_i(u,\mu_j(a,b))=~&\sum_{i+j= n}r_i(r_j(u,a),b),\label{formal3} \\
\sum_{i+j+k = n} \mu_i(T_{1,j} (u)  ,  T_{1,k} (v) ) =~& \sum_{i+j+k  = n} T_{1,i} \big(   l_j(T_{1,k} (u) , v) + r_j(u, T_{1,k} (v) ) \big), \label{formal4} \\
\sum_{i+j+k = n} \mu_i(T_{2,j} (u)  ,  T_{2,k} (v) ) =~& \sum_{i+j+k  = n} T_{2,i} \big(   l_j(T_{2,k} (u) , v) + r_j(u, T_{2,k} (v) ) \big), \label{formal5} \\
\sum_{i+j+k = n} \mu_i(T_{1,j} (u)  ,  T_{2,k} (v) ) + \sum_{i+j+k = n} \mu_i(T_{2,j} (u)  ,  T_{1,k} (v) ) =~& \sum_{i+j+k  = n} T_{1,i} \big(   l_j(T_{2,k} (u) , v) + r_j(u, T_{2,k} (v) ) \big) \label{formal6}  \\
&+ \sum_{i+j+k  = n} T_{2,i} \big(   l_j(T_{1,k} (u) , v) + r_j(u, T_{1,k} (v) ) \big), \nonumber
\end{align}
for $a, b, c \in A$, $u , v \in M$ and $n = 0, 1, \ldots$ . These equations are obviously hold for $n=0$ as the triple $(A = (A,\mu),M=(M,l,r),(T_1,T_2))$ is a compatible $\mathcal{O}$-operator algebra. However, for $n =1$, we get
\begin{small}
\begin{align}
\mu_1(a,b) \cdot c+\mu_1(a \cdot b, c)=~& a \cdot \mu_1(b,c)+\mu_1(a,b \cdot c),\label{com-def1}\\
	 \mu_1(a, b) \cdot u +l_1(a \cdot b, u)=~& a \cdot l_1(b, u)+l_1(a,b \cdot u),\label{com-def2}\\
	 r_1(a \cdot u, b) + l_1(a , u) \cdot b =~& a \cdot r_1( u, b)+l_1(a,u \cdot b),\label{com-def2p}\\
	 u \cdot \mu_1(a,b) + r_1(u, a \cdot b)=~& r_1(u,a) \cdot b +r_1(u \cdot a,b),\label{com-def3}
\end{align}	
\begin{align}
\mu_1(T_{1} (u), T_{1} (v))+T_{1} (u) \cdot T_{1,1}(v) + T_{1,1}(u) \cdot T_{1}(v) =~& T_{1} \big(  T_{1,1}(u) \cdot v + u \cdot T_{1,1}(v) + l_1(T_{1}(u), v) +r_1(u, T_{1}(v))  \big) \label{com-def4}\\
& \qquad \qquad + T_{1,1} \big( T_{1}(u) \cdot v  + u \cdot T_{1} (v) \big), \nonumber \\
\mu_1(T_{2} (u), T_{2} (v))+T_{2} (u) \cdot T_{2,1}(v) + T_{2,1}(u) \cdot T_{2}(v) =~& T_{2} \big(  T_{2,1}(u) \cdot v + u \cdot T_{2,1}(v) + l_1(T_{2}(u), v) +r_1(u, T_{2}(v))  \big) \label{com-def5} \\
& \qquad \qquad + T_{2,1} \big( T_{2}(u) \cdot v  + u \cdot T_{2} (v) \big), \nonumber
\end{align}
\begin{align}
\mu_1 &(T_{1} (u), T_{2} (v))+T_{1} (u) \cdot T_{2,1}(v) + T_{1,1}(u) \cdot T_{2}(v)   +  \mu_1(T_{2} (u), T_{1} (v))+T_{2} (u) \cdot T_{1,1}(v) + T_{2,1}(u) \cdot T_{1}(v) \label{com-def6}  \\
&= T_{1} \big(  T_{2,1}(u) \cdot v + u \cdot T_{2,1}(v) + l_1(T_{2}(u), v) +r_1(u, T_{2}(v))  \big)
+ T_{1,1} \big( T_{2}(u) \cdot v  + u \cdot T_{2} (v) \big) \nonumber \\
& \quad + T_{2} \big(  T_{1,1}(u) \cdot v + u \cdot T_{1,1}(v) + l_1(T_{1}(u), v) +r_1(u, T_{1}(v))  \big) 
+ T_{2,1} \big( T_{1}(u) \cdot v  + u \cdot T_{1} (v) \big). \nonumber 
\end{align}
\end{small}
Combining all these relations, we get that
\begin{align*}
\delta_\mathrm{cOA} (\mu_1 + l_1 + r_1 , (T_{1,1}, T_{2,1})) = 0.
\end{align*}
In other words, $( (\mu_1+l_1+r_1),(T_{1,1},T_{2,1}) )$ is a $2$-cocycle in the cohomology complex of the compatible $\mathcal{O}$-operator algebra $(A, M, (T_1, T_2))$. This is called the infinitesimal of the deformation $(\mu_t,l_t,r_t,T_{1,t},T_{2,t})$.

\begin{defn}
Let $(\mu_t, l_t,r_t, T_{1,t},T_{2,t})$ and $(\mu'_t ,l'_t,r'_t,T'_{1,t},T'_{2,t})$ be two deformations of the compatible $\mathcal{O}$-operator algebra $(A= (A,\mu), M=(M,l,r),(T_1,T_2))$. These two deformations are said to be equivalent if there exist formal maps $\phi_t = \sum_{i=0}^\infty t^i \phi_i$ and $\psi_t = \sum_{i=0}^\infty t^i \psi_i$, where $\phi_i \in \mathrm{Hom}(A,A)$, $\psi_i \in \mathrm{Hom}(M,M)$ with $\phi_0 = \mathrm{id}_A$ and $\psi_0 = \mathrm{id}_M$, such that the pair $(\phi_t, \psi_t)$ is a morphism of compatible $\mathcal{O}$-operator algebras from $\big( (A[[t]],\mu_t),  (M[[t]],l_t,r_t),(T_{1,t},T_{2,t}) \big)$ to $\big( (A[[t]],\mu'_t), (M[[t]],l'_t,r'_t),(T'_{1,t},T'_{2,t}) \big)$.
\end{defn}

Thus, the equivalence of $(\mu_t, l_t,r_t, T_{1,t},T_{2,t})$ and $(\mu'_t ,l'_t,r'_t,T'_{1,t},T'_{2,t})$ amounts the following system of equations
\begin{align*}
\sum_{i+j = n} \phi_i (\mu_j (a, b) ) =~& \sum_{i+j+k = n} \mu_i' \big(  \phi_j (a) , \phi_k (b) \big), \\
\sum_{i+j = n} \psi_i (l_j (a, u) ) =~& \sum_{i+j+k = n} l_i' \big(  \phi_j (a) , \psi_k (u) \big), \\
\sum_{i+j = n} \psi_i (r_j (u,a) ) =~& \sum_{i+j+k = n} r_i' \big(  \psi_j (u) , \phi_k (a) \big),\\
\sum_{i+j = n} \phi_i \circ T_{1, j} =~& \sum_{i+j=n} T'_{1, i} \circ \psi_j \\
\sum_{i+j = n} \phi_i \circ T_{2, j} =~& \sum_{i+j=n} T'_{2, i} \circ \psi_j,
\end{align*}
for $a, b \in A$, $u \in M$ and $n =0, 1, \ldots$ . These are obviously hold for $n=0$. However, for $n=1$, we get
\begin{align*}
\mu_1 (a, b) - \mu_1' (a, b) =~& a \cdot \phi_1 (b) - \phi_1 (a \cdot b) + \phi_1 (a) \cdot b,\\
l_1 (a, u) - l_1' (a, u) =~& a \cdot \psi_1 (u) - \psi_1 (a \cdot u) + \phi_1 (a) \cdot u,\\
r_1 (u, a) - r_1' (u, a) =~& u \cdot \phi_1 (a) - \psi_1 (u \cdot a) + \psi_1 (u) \cdot a,\\
T_{1,1} - T'_{1,1} =~& T_1' \circ \psi_1 - \phi_1 \circ T_{1}, \\
T_{2,1} - T'_{2,1} =~& T_2' \circ \psi_1 - \phi_1 \circ T_{2}.
\end{align*}
Combining these relations, we get that
\begin{align*}
\big(  \mu_1 + l_1 + r_1 , (T_{1,1}, T_{2,1}) \big) - \big(  \mu'_1 + l'_1 + r'_1 , (T'_{1,1}, T'_{2,1}) \big) = \delta_\mathrm{cOA} \big( (\phi_1, \psi_1) \big).
\end{align*}
As a consequence, we get the following.

\begin{thm}
Let $(A,,M,(T_1,T_2))$ be a compatible $\mathcal{O}$-operator algebra. Then the infinitesimal in a formal deformation of this compatible $\mathcal{O}$-operator algebra is a $2$-cocycle in the cohomology complex of $(A,M,(T_1,T_2))$. Moreover, the corresponding cohomology class depends only on the equivalence class of the formal deformation.
\end{thm}

\medskip

\section{Compatible dendriform algebras}\label{sec-comp-dend}
In this section, we first consider compatible dendriform algebras introduced in \cite{LSB19} and find their relations with compatible $\mathcal{O}$-operators. Then we introduce cohomology of compatible dendriform algebras and relate with the cohomology of compatible associative algebras and compatible $\mathcal{O}$-operators.

\begin{defn} \cite{L01} A dendriform algebra is a triple $(D, \prec, \succ)$ consisting of a vector space $D$ together with bilinear operations $\prec, \succ : D \otimes D \rightarrow D$ satisfying for $x, y, z \in D$,
\begin{align}
(x\prec y)\prec z=~& x\prec(y\prec z+y\succ z), \label{dend-id1}\\
 (x\succ y)\prec z=~& x\succ (y\prec z),\\
(x\prec y+x\succ y)\succ z=~& x\succ(y\succ z). \label{dend-id3}
\end{align}
\end{defn}

\begin{remark}
(i) Let $(D, \prec, \succ)$ be a dendriform algebra. It follows from the identities (\ref{dend-id1})-(\ref{dend-id3}) that the bilinear operation $x \star y :=  x \prec y + x \succ y$ is associative. In other words, $(D, \star)$ is an associaative algebra.

(ii) Let $T: M \rightarrow A$ be an $\mathcal{O}$-operator on $M$ over the algebra $A$. Then $(M, \prec_T, \succ_T)$ is a dendriform algebra, where $u \prec_T v := u \cdot T(v)$ and $u \succ_T v := T(u) \cdot v$, for $u, v \in M$.
\end{remark}

Dendriform algebras arise naturally in the arithmetic of planar binary trees and shuffle algebras \cite{L01,G12}. See also \cite{dendr1,lod-val} for some interesting results about dendriform algebras.

\begin{defn}
A compatible dendriform algebra is a quintuple $(D, \prec_1, \succ_1, \prec_2, \succ_2)$ consisting of a vector space $D$ together with four bilinear operations
$\prec_1, \succ_1, \prec_2, \succ_2 : D \otimes D \rightarrow D$
so that $(D, \prec_1, \succ_1)$ and $(D, \prec_2, \succ_2)$ are both dendriform algebras satisfying for $x, y, z \in D$,
\begin{align}
(x\prec_1 y)\prec_2 z + (x\prec_2 y)\prec_1 z =~& x\prec_2(y\prec_1 z+y\succ_1 z) + x\prec_1(y\prec_2 z+y\succ_2 z), \label{comp-dend1}\\
 (x\succ_1 y)\prec_2 z + (x\succ_2 y)\prec_1 z =~& x\succ_2 (y\prec_1 z) + x\succ_1 (y\prec_2 z) ,\\
(x\prec_1 y+x\succ_1 y)\succ_2 z + (x\prec_2 y+x\succ_2 y)\succ_1 z =~& x\succ_2(y\succ_1 z) + x\succ_1 (y\succ_2 z). \label{comp-dend3}
\end{align}
\end{defn}

\begin{remark}
The compatibility conditions (\ref{comp-dend1})-(\ref{comp-dend3}) in a compatible dendriform algebra is equivalent to the fact that $(D, \lambda \prec_1 + \eta \prec_2 , \lambda \succ_1 + \eta \succ_2)$ is a dendriform algebra, for any $\lambda, \eta \in {\bf k}.$
\end{remark}

Let $(D, \prec_1, \succ_1, \prec_2, \succ_2)$ and $(D', \prec'_1, \succ'_1, \prec'_2, \succ'_2)$ be two compatible dendriform algebras. A morphism between them is a linear map $\psi : D \rightarrow D'$ which preserve the respective structure maps.

\begin{prop}\label{dend-tot}
Let $(D, \prec_1, \succ_1, \prec_2, \succ_2)$ be a compatible dendriform algebra. Then the triple $(D, \star_1, \star_2)$ is a compatible associative algebra, where $x \star_1 y := x\prec_1 y + x \succ_1 y$ and $x \star_2 y := x \prec_2 y + x \succ_2 y$. (This is called the total compatible associative algebra, denoted by $D^\mathrm{Tot}$).
\end{prop}

\begin{proof}
Since $(D, \prec_1, \succ_1)$ and $(D, \prec_2, \succ_2)$ are both dendriform algebras, it follows that $(D, \star_1)$ and $(D, \star_2)$ are both associative algebras. On the other hand, by adding the respective sides of the identities (\ref{comp-dend1})-(\ref{comp-dend3}), we get that
\begin{align*}
(x \star_1 y) \star_2 z + (x \star_2 y) \star_1 z = x \star_2 (y \star_1 z) + x \star_1 (y \star_2 z), \text{ for } x, y, z \in D.
\end{align*}
This shows that $(D, \star_1, \star_2)$ is a compatible associative algebra.
\end{proof}

It is well-known that dendriform algebras are related to pre-Lie algebras \cite{aguiar-pre-dend}. In the same way, compatible dendriform algebras are related to compatible pre-Lie algebras. We first recall some definitions about pre-Lie algebras \cite{aguiar-pre-dend,LSB19}.

\begin{defn}
(i) A (left) pre-Lie algebra is a pair $(P, \diamond)$ consisting of a vector space $P$ together with a bilinear operation $\diamond : P \otimes P \rightarrow P$ satisfying
\begin{align*}
( x \diamond y ) \diamond z - x \diamond (y \diamond z) = ( y \diamond x ) \diamond z - y \diamond (x \diamond z), ~ \text{ for } x, y, z \in P.
\end{align*}

(ii) The `subadjacent Lie algebra' of a pre-Lie algebra $(P, \diamond)$ is the Lie algebra $(P, [~,~]_\diamond)$, where
\begin{align*}
[x,y]_\diamond := x \diamond y - y \diamond x, \text{ for } x, y \in P.
\end{align*}

(iii) A compatible (left) pre-Lie algebra is a triple $(P, \diamond_1, \diamond_2)$ in which $(P, \diamond_1)$ and $(P, \diamond_2)$ are both (left) pre-Lie algebras satisfying additionally
\begin{align*}
( x \diamond_1 y ) \diamond_2 z &+ ( x \diamond_2 y ) \diamond_1 z - x \diamond_2 (y \diamond_1 z)  - x \diamond_1 (y \diamond_2 z)  \\&= ( y \diamond_1 x ) \diamond_2 z + ( y \diamond_2 x ) \diamond_1 z -  y \diamond_2 (x \diamond_1 z) - y \diamond_1 (x \diamond_2 z), ~ \text{ for } x, y, z \in P.
\end{align*}
\end{defn}

The following result is straightforward. Hence we omit the details.

\begin{prop}
(i) Let $(D, \prec_1, \succ_1, \prec_2, \succ_2)$ be a compatible dendriform algebra. Then $(D, \diamond_1, \diamond_2)$ is a compatible pre-Lie algebra, where $x \diamond_1 y := x \succ_1 y - y \prec_1 x$ and $x \diamond_2 y := x \succ_2 y - y \prec_2 x$, for $x, y \in D$.

(ii) Let $(P, \diamond_1, \diamond_2)$ be a compatible (left) pre-Lie algebra. Then $(P, [~,~]_{\diamond_1}, [~,~]_{\diamond_1})$ is a {compatible Lie algebra}.
\end{prop}

With all these constructions, we have the following commutative diagram
\[
\xymatrix{
 \text{ comp. dendriform algebra} \ar[r]^{\mathrm{total}} \ar[d] & \text{ comp. associative algebra} \ar[d]^{\text{skew-symmetrization}} \\
 \text{ comp. pre-Lie algebra} \ar[r]_{\text{subadjacent}} &  \text{ comp. Lie algebra}. \\
}
\]

The following result has been proved in \cite{LSB19}.

\begin{prop}\label{ind-dend}
(i) Let $A$ be an associative algebra and $M$ be an $A$-bimodule. Let $(T_1, T_2)$ be a compatible $\mathcal{O}$-operator on $M$ over the algebra $A$. Then $(M, \prec_{T_1}, \succ_{T_1}, \prec_{T_2}, \succ_{T_2})$ is a compatible dendriform algebra, denoted by $M^\mathrm{Ind}$, where
\begin{align*}
u \prec_{T_1} v := u \cdot T_1 (v), ~~
u \succ_{T_1} v := T_1 (u) \cdot v, ~~
u \prec_{T_2} v := u \cdot T_2 (v) ~ \text{ and } ~
u \succ_{T_2} v := T_2 (u) \cdot v, ~\text{for } u, v \in M.
\end{align*}

(ii) Let $(T_1, T_2)$ and $(T'_1, T'_2)$ be  two compatible $\mathcal{O}$-operators on $M$ over the algebra $A$. If $(\phi, \psi)$ is a morphism of compatible $\mathcal{O}$-operators from $(T_1, T_2)$ to $(T'_1, T'_2)$ then $\psi : M \rightarrow M$ is a morphism of compatible dendriform algebras from $(M, \prec_{T_1}, \succ_{T_1}, \prec_{T_2}, \succ_{T_2})$ to $(M, \prec_{T'_1}, \succ_{T'_1}, \prec_{T'_2}, \succ_{T'_2})$.
\end{prop}

\begin{proof}
(i) Since $T_1$ is an $\mathcal{O}$-operator, it follows that $(M, \prec_{T_1}, \succ_{T_1})$ is a dendriform algebra. Similarly, $T_2$ is an $\mathcal{O}$-operator implies that $(M, \prec_{T_2}, \succ_{T_2})$ is a dendriform algebra. Finally, $T_1 + T_2$ is an $\mathcal{O}$-operator implies that $(M, \prec_{T_1 + T_2} = \prec_{T_1} + \prec_{T_2 },  \succ_{T_1 + T_2} = \succ_{T_1} + \succ_{T_2})$ is a dendriform algebra. This shows that $(M, \prec_{T_1}, \succ_{T_1}, \prec_{T_2}, \succ_{T_2})$ is a compatible dendriform algebra.

(ii) For any $u, v \in M$, we have
\begin{align*}
\psi (u \prec_{T_1} v) = \psi (u \cdot T_1 (v)) = \psi (u) \cdot T_1' (\psi (v)) = \psi (u) \prec_{T_1'} \psi (v), \\
\psi (u \succ_{T_1} v) = \psi (T_1 (u) \cdot v) = T_1' (\psi (u)) \cdot \psi (v) = \psi (u) \succ_{T_1'} \psi (v).
\end{align*}
Similarly, we have $\psi ( u \prec_{T_2} v) = \psi (u) \prec_{T_2'} \psi (v)$ and $\psi ( u \succ_{T_2} v) = \psi (u) \succ_{T_2'} \psi (v)$. This proves the result.
\end{proof}

\begin{remark}
Let $(T_1,T_2)$ be a compatible $\mathcal{O}$-operator on $M$ over the algebra $A$. Then $(M, \diamond_1, \diamond_2)$ is a compatible pre-Lie algebra, where
\begin{align*}
u \diamond_1 v := u \succ_1 v - v \prec_1 u ~~~ \text{ and } ~~~ u \diamond_2 v := u \succ_2 v - v \prec_2 u, \text{ for } u, v \in M.
\end{align*}
\end{remark}


\medskip

\noindent {\bf Cohomology of compatible dendriform algebras.}
Here we define cohomology theory of compatible dendriform algebras. We first recall certain combinatorial maps defined in \cite{Das11802}. Let $C_n$ be the set of first $n$ natural numbers. We will treat elements of $C_n$ as symbols. Hence we write $C_n = \{[1], [2],\ldots, [n]\}$. For each $m, n \geq 1$ and $1 \leq i \leq m$, we put the elements of $C_{m+n-1}$ into $m$ many boxes as follows:
\begin{align*}
\framebox{$[1]$} \quad \framebox{$[2]$} \quad \cdots \quad \framebox{$[i-1]$} \quad \framebox{$[i]$ \quad $[i+1]$ $\cdots$ $[i+n-1]$} \quad \framebox{$[i+n]$} \quad \cdots \quad \framebox{$[m+n-1]$}~.
\end{align*}
Observe that there is exactly one element in each box except the $i$-th box which contains $n$ elements. We define maps $R_{m; i, n} : C_{m+n-1} \rightarrow C_m$ and $S_{m; i, n} : C_{m+n-1} \rightarrow {\bf k}[C_n]$ by
\begin{align*}
R_{m;i, n} ([r]) =~& k ~~~ \text{ if } [r] \text{ lies in the }k\text{-th box},\\
S_{m;i, n} ([r]) =~& \begin{cases}
[1]+ \cdots + [n] ~ &\text{ if } [r] \not\in \{ [i], [i+1], \ldots, [i+n-1]\}\\
&~~~ (\text{i.e. if } r \text{ doesn't lie in the }i\text{-th box})\\
[r-i+1] ~ &\text{ if } [r] \in \{ [i], [i+1], \ldots, [i+n-1]\}.
\end{cases}
\end{align*}
For any vector space $D$ (not necessarily a dendriform algebra), one may now define maps (called partial compositions)
\begin{align*}
\circ_i :  \mathrm{Hom} (\mathbf{k}[C_{m}]\otimes D^{\otimes m}, D) ~ \otimes ~  \mathrm{Hom} (\mathbf{k}[C_{n}]\otimes D^{\otimes n}, D) \rightarrow  \mathrm{Hom} (\mathbf{k}[C_{m+n-1}]\otimes D^{\otimes m+n-1}, D) ~~ \text{ by }
\end{align*}
\begin{align*}
(f \circ_i g) ([r]; x_1, \ldots, x_{m+n-1} ) = f \big(  R_{m;i, n}[r]; x_1, \ldots, x_{i-1}, g \big( S_{m;i,n}[r]; x_i, \ldots, x_{i+n-1}  \big), x_{i+n},\ldots, x_{m+n-1}  \big),
\end{align*}
for $f \in \mathrm{Hom} (\mathbf{k}[C_{m}]\otimes D^{\otimes m}, D)$, $g \in \mathrm{Hom} (\mathbf{k}[C_{n}]\otimes D^{\otimes n}, D)$ and $[r] \in C_{m+n-1}$.
It has been observed in \cite{Das11802} that the collection $\{  \mathrm{Hom} (\mathbf{k}[C_{n}]\otimes D^{\otimes n}, D)  \}_{n \geq 1}$ of spaces with the above partial compositions forms a nonsymmetric operad. Therefore, there is a degree $-1$ graded Lie bracket on $\oplus_{n \geq 1} \mathrm{Hom} (\mathbf{k}[C_{n}]\otimes D^{\otimes n}, D) $ given by
\begin{align*}
\{ \! \! \{ f, g \} \! \! \}  = \sum_{i=1}^m (-1)^{(i-1)(n-1)} f \circ_i g ~-~ (-1)^{(m-1)(n-1)} \sum_{i=1}^n (-1)^{(i-1)(m-1)} g \circ_i f.
\end{align*}

\medskip

Let $(D, \prec_1, \succ_1, \prec_2, \succ_2)$ be a compatible dendriform algebra. Define $\pi_1, \pi_2 \in \mathrm{Hom}({\bf k}[C_2]\otimes D^{\otimes 2}, D )$ by
\begin{align*}
\pi_1 ([r];x, y) = \begin{cases}
x \prec_1 y &\text{ if } [r] =[1] \\
x \succ_1 y &\text{ if } [r] =[2]
\end{cases}
~~~ \text{ and } ~~~
\pi_2 ([r];x, y) = \begin{cases}
x \prec_2 y &\text{ if } [r] =[1] \\
x \succ_2 y &\text{ if } [r] =[2].
\end{cases}
\end{align*}
Since $(\prec_1, \succ_1)$ defines a dendriform structure on $D$, it follows that $\{ \! \! \{ \pi_1, \pi_1 \} \! \! \} = 0$. Similarly, we have $\{ \! \! \{ \pi_2 , \pi_2 \} \! \! \} = 0$. Finally, the compatibility conditions (\ref{comp-dend1})-(\ref{comp-dend3}) is equivalent to $\{ \! \! \{ \pi_1, \pi_2 \} \! \! \} = 0$. We are now in a position to define the cohomology of the compatible dendriform algebra $(D, \prec_1, \succ_1, \prec_2, \succ_2)$.

For each $n \geq 1$, we define an abelian group $C^n_\mathrm{cDend}(D,D)$ by
\begin{align*}
C^n_\mathrm{cDend} (D,D) = \underbrace{ \mathrm{Hom} (\mathbf{k}[C_{n}]\otimes D^{\otimes n}, D) \oplus \cdots \oplus \mathrm{Hom} (\mathbf{k}[C_{n}]\otimes D^{\otimes n}, D)   }_{n\text{ summand}}
\end{align*}
and define a map $\delta_\mathrm{cDend} : C^n_\mathrm{cDend} (D,D) \rightarrow C^{n+1}_\mathrm{cDend} (D,D)$ by
\begin{align*}
\delta_\mathrm{cDend} ((f_1, \ldots, f_n)) = (-1)^{n-1} \big(   \{ \! \! \{ \pi_1, f_1 \} \! \! \}, \ldots, \underbrace{ \{ \! \! \{ \pi_1, f_i \} \! \! \} + \{ \! \! \{ \pi_2, f_{i-1} \} \! \! \} }_{i\text{-th place}} , \ldots, \{ \! \! \{ \pi_2, f_n \} \! \! \} \big).
\end{align*}

\begin{prop}
The map $\delta_\mathrm{cDend}$ satisfies $(\delta_\mathrm{cDend})^2 =0$.
\end{prop}

\begin{proof}
For any $(f_1, \ldots, f_n) \in C^n_\mathrm{cDend} (D,D)$, we have
\begin{align*}
    &(\delta_\mathrm{cDend})^2 ((f_1,\cdots,f_{n}))\\
    &=(-1)^{n-1} ~ \big(   \{ \! \! \{ \pi_1, f_1 \} \! \! \}, \ldots, \underbrace{ \{ \! \! \{ \pi_1, f_i \} \! \! \} + \{ \! \! \{ \pi_2, f_{i-1} \} \! \! \} }_{i\text{-th place}} , \ldots, \{ \! \! \{ \pi_2, f_n \} \! \! \} \big)\\
    &=- \big( \{ \! \! \{ \pi_1,\{ \! \! \{ \pi_1,f_1 \} \! \! \}  \} \! \! \}, \{ \! \! \{ \pi_1, \{ \! \! \{ \pi_1,f_2 \} \! \! \} \} \! \! \} + \{ \! \! \{ \pi_1,\{ \! \! \{ \pi_2,f_1 \} \! \! \}  \} \! \! \} + \{ \! \! \{ \pi_2,\{ \! \! \{ \pi_1,f_1 \} \! \! \}\} \! \! \}  ,\ldots,\\
    & ~ \qquad \underbrace{\{ \! \! \{ \pi_1,\{ \! \! \{ \pi_1,f_{i}\} \! \! \} \} \! \! \} + \{ \! \! \{ \pi_1,\{ \! \! \{ \pi_2,f_{i-1} \} \! \! \} \} \! \! \} + \{ \! \! \{ \pi_2,\{ \! \! \{ \pi_1,f_{i-1} \} \! \! \} \} \! \! \} + \{ \! \! \{ \pi_2, \{ \! \! \{ \pi_2,f_{i-2} \} \! \! \} \} \! \! \} }_{3\leq i\leq n-1},\ldots,\\
& ~ \qquad \{ \! \! \{ \pi_1,\{ \! \! \{ \pi_2,f_n \} \! \! \} \} \! \! \} + \{ \! \! \{ \pi_2,\{ \! \! \{ \pi_1,f_n \} \! \! \} \} \! \! \} +  \{ \! \! \{ \pi_2, \{ \! \! \{ \pi_2,f_{n-1} \} \! \! \} \} \! \! \},
\{ \! \! \{ \pi_2,\{ \! \! \{ \pi_2,f_n \} \! \! \} \} \! \! \} \big)\\
    &=(0,0,\ldots,0)   \qquad (\text{as } \{ \! \! \{ \pi_1, \pi_1 \} \! \! \} = \{ \! \! \{ \pi_2, \pi_2 \} \! \! \} = \{ \! \! \{ \pi_1, \pi_2 \} \! \! \} = 0).
\end{align*}
This completes the proof.
\end{proof}

The cohomology groups of the cochain complex $\{ C^\bullet_\mathrm{cDend} (D,D), \delta_\mathrm{cDend} \}$ are called the cohomology of the compatible dendriform algebra $(D, \prec_1, \succ_1, \prec_2, \succ_2 )$.

\medskip

\medskip

\noindent {\bf Relation with the cohomology of compatible associative algebras.} Here we find a morphism from the cohomology of a compatible dendriform algebra to the cohomology of the corresponding total compatible associative algebra.

\medskip

Let $(D, \prec_1, \succ_1, \prec_2, \succ_2 )$ be a compatible dendriform algebra. Then we have seen in Proposition \ref{dend-tot} that $D^\mathrm{Tot} = (D, \star_1, \star_2)$ is a compatible associative algebra. Let $\{ C^\bullet_\mathrm{cAss}(D^\mathrm{Tot}, D^\mathrm{Tot}), \delta_{cAss}  \}$ be the cochain complex defining the cohomology of the compatible associative algebra $D^\mathrm{Tot}$ with coefficients in the adjoint compatible bimodule. For each $n \geq 1$, we define a map
\begin{center}
$\Phi_n : C^n_\mathrm{cDend} (D,D) \rightarrow C^n_\mathrm{cAss} (D^\mathrm{Tot}, D^\mathrm{Tot} )$  ~~ by
\end{center}
\begin{align*}
\Phi_n ((f_1, \ldots, f_n )) = (f_1^\mathrm{Tot}, \ldots, f_n^\mathrm{Tot}), \text{ where } f_i^\mathrm{Tot} (x_1, \ldots, x_n ) = \sum_{r=1}^n f_i ([r]; x_1, \ldots, x_n).
\end{align*}
It has been shown in \cite{Das11802} that $\delta^1_\mathrm{Ass} (f_i^\mathrm{Tot}) = (-1)^{n-1} \{ \! \! \{ \pi_1, f_i \} \! \! \}^\mathrm{Tot}$ and $\delta^2_\mathrm{Ass} (f_i^\mathrm{Tot}) = (-1)^{n-1} \{ \! \! \{ \pi_2, f_i \} \! \! \}^\mathrm{Tot}$. Here $\delta^1_\mathrm{Ass}$ (resp. $\delta^2_\mathrm{Ass}$) denotes the Hochschild coboundary operator of the associative algebra $(D, \star_1)$ ~\big(resp. $(D, \star_2)$\big) with coefficients in itself. Therefore, we have
\begin{align*}
(\delta_\mathrm{cAss} \circ \Phi_n) ((f_1, \ldots, f_n ))
=~& \delta_\mathrm{cAss} (  f_1^\mathrm{Tot}, \ldots, f_n^\mathrm{Tot} ) \\
=~& \big( \delta^1_\mathrm{Ass} ( f_1^\mathrm{Tot}), \ldots, \delta^1_\mathrm{Ass} ( f_i^\mathrm{Tot}) + \delta^2_\mathrm{Ass} ( f_{i-1}^\mathrm{Tot}) , \ldots, \delta^2_\mathrm{Ass} ( f_n^\mathrm{Tot})\big) \\
=~& (-1)^{n-1} \big( \{ \! \! \{ \pi_1, f_1 \} \! \! \}^\mathrm{Tot},  \ldots, \{ \! \! \{ \pi_1, f_i \} \! \! \}^\mathrm{Tot} + \{ \! \! \{ \pi_2, f_{i-1} \} \! \! \}^\mathrm{Tot} , \ldots, \{ \! \! \{ \pi_2, f_n \} \! \! \}^\mathrm{Tot}\big) \\
=~& (\Phi_{n+1} \circ \delta_\mathrm{cDend}) ((f_1, \ldots, f_n )).
\end{align*}
As a consequence, we get the following.

\begin{thm}
Let $(D, \prec_1, \succ_1, \prec_2, \succ_2 )$ be a compatible dendriform algebra and let $D^\mathrm{Tot} = (D, \star_1, \star_2)$ be the corresponding total compatible associative algebra. Then the collection $\{ \Phi_n \}_{n \geq 1}$ of maps induces a morphism $H^\bullet_\mathrm{cDend}(D,D) \rightarrow H^\bullet_\mathrm{cAss} (D^\mathrm{Tot},D^\mathrm{Tot})$ between corresponding cohomologies.
\end{thm}

\medskip

\noindent {\bf Relation with the cohomology of compatible $\mathcal{O}$-operators.} Let $A$ be an associative algebra and $M$ be an $A$-bimodule. Suppose $(T_1, T_2)$ is a compatible $\mathcal{O}$-operator on $M$ over the algebra $A$. Then we have seen in Proposition \ref{ind-dend} that $M$ carries a compatible dendriform algebra structure (denoted by $M^\mathrm{Ind}$). Let $\pi_1, \pi_2 \in \mathrm{Hom}({\bf k}[C_2]\otimes M^{\otimes 2}, M)$ be the elements given by
\begin{align*}
\pi_1 ([r];u,v) = \begin{cases}
u \prec_{T_1} v &\text{ if } [r] =[1] \\
u \succ_{T_1} v &\text{ if } [r] =[2]
\end{cases}
~~~ \text{ and } ~~~
\pi_2 ([r];u,v) = \begin{cases}
u \prec_{T_2} v &\text{ if } [r] =[1] \\
u \succ_{T_2} v &\text{ if } [r] =[2].
\end{cases}
\end{align*}
In the following, we find a morphism from the cohomology of the compatible $\mathcal{O}$-operator $(T_1, T_2)$ to the cohomology of the compatible dendriform algebra $M^\mathrm{Ind}$.

For each $n \geq 1$, we define a map
$\Psi_n : C^n_\mathrm{cO}(M,A) \rightarrow C^{n+1}_\mathrm{cDend}(M^\mathrm{Ind}, M^\mathrm{Ind})$ by
\begin{align*}
\Psi_n ((f_1, \ldots, f_{n+1})) := (f_1^\mathrm{Ind}, \ldots, f_{n+1}^\mathrm{Ind}), ~~\mathrm{ where}
\end{align*}
\begin{align*}
(f_i^\mathrm{Ind}) ([r]; u_1, \ldots, u_{n+1}) = \begin{cases} (-1)^{n+1} ~u_1 \cdot f_i (u_2, \ldots, u_{n+1}) & \text{ if }~ [r] = [1] \\
f_i(u_1, \ldots, u_n) \cdot u_{n+1} & \text{ if }~ [r] = [n+1] \\
0 & \text{ if }~ [r] \neq [1], [n+1] . \end{cases}
\end{align*}
It has been shown in \cite[Subsec 3.1]{D20} that $\{ \! \! \{ \pi_1, f_i^\mathrm{Ind} \} \! \! \} = (-1)^n ~[T_1, f_i ]^\mathrm{Ind}$ and $\{ \! \! \{ \pi_2, f_i^\mathrm{Ind} \} \! \! \} = (-1)^n ~[T_2, f_i ]^\mathrm{Ind}$, for all $i=1, \ldots, n+1$. Therefore, we have
\begin{align*}
~&(\delta_\mathrm{cDend} \circ \Psi_n) ((f_1, \ldots, f_{n+1})) = \delta_\mathrm{cDend} ( f_1^\mathrm{Ind}, \ldots, f_{n+1}^\mathrm{Ind} ) \\
&= (-1)^n \big(  \{ \! \! \{ \pi_1, f_1^\mathrm{Ind} \} \! \! \}, \ldots, \{ \! \! \{ \pi_1, f_i^\mathrm{Ind} \} \! \! \} + \{ \! \! \{ \pi_2, f_{i-1}^\mathrm{Ind} \} \! \! \}, \ldots, \{ \! \! \{ \pi_2, f_{n+1}^\mathrm{Ind} \} \! \! \} \big) \\
&= \big(  [T_1, f_1]^\mathrm{Ind}, \ldots, [T_1, f_i]^\mathrm{Ind}  + [T_2, f_{i-1}]^\mathrm{Ind}, \ldots, [T_2, f_{n+1}]^\mathrm{Ind} \big) \\
&=(\Psi_{n+1} \circ \delta_{(T_1, T_2)} ) ((f_1, \ldots, f_{n+1})).
\end{align*}
As a consequence, we get the following.
\begin{thm}
Let $(T_1,T_2)$ be a compatible $\mathcal{O}$-operator on $M$ over the algebra $A$. Then the collection $\{ \Psi_n \}_{n \geq 1}$ of maps induces a morphism $H^\bullet_{(T_1, T_2)} (M,A) \rightarrow H^{\bullet +1}_\mathrm{cDend}(M^\mathrm{Ind}, M^\mathrm{Ind})$ from the cohomology of the compatible $\mathcal{O}$-operator $(T_1, T_2)$ to the cohomology of the induced compatible dendriform algebra $M^\mathrm{Ind}$.
\end{thm}

\medskip

\noindent {\bf {Acknowledgements.}}
A. Das thanks the Department of Mathematics, IIT Kharagpur (India) for providing the beautiful academic atmosphere. The work of S. Guo is supported by the NSF of China (No. 12161013).     Y. Qin is financed by  the NSF of China  (No.    12071137)   and  by  STCSM  (No. 18dz2271000).




\end{document}